\documentclass[preprint,12pt]{elsarticle}
\usepackage[margin=2.5cm]{geometry}

\usepackage{graphicx}
\usepackage{amssymb}
\usepackage{eqsection}
\usepackage{appendix}

\usepackage[centertags]{amsmath}
\usepackage{amsfonts}
\usepackage{amssymb}
\usepackage{amsthm}
\usepackage[symbol*]{footmisc}
\usepackage{tikz}
\usepackage{tkz-euclide}
\DefineFNsymbolsTM{myfnsymbols}{
  \textasteriskcentered *
  \textdagger    \dagger
  \textdaggerdbl \ddagger
}
\setfnsymbol{myfnsymbols}
\usepackage{xcolor}

\newtheorem{thm}{Theorem}[section]
\newtheorem{prop}[thm]{Proposition}
\newtheorem{lem}[thm]{Lemma}
\newtheorem{oss}[thm]{Remark}

\newtheorem{cor}[thm]{Corollary}

\def\RR{{\mathbb{R}}}

\def\NN{{\mathbb{N}}}
\newcommand{\eps}{\varepsilon}

\definecolor{green(ryb)}{rgb}{0.173, 0.627, 0.173}
\newcommand{\norm}[1]{\left\lVert #1 \right\rVert}

\makeatletter
\def\ps@pprintTitle{%
   \let\@oddhead\@empty
   \let\@evenhead\@empty
   \let\@oddfoot\@empty
   \let\@evenfoot\@oddfoot
}
\makeatother
\begin{document}

\begin{frontmatter}


\title{Exact controllability to eigensolutions for evolution equations of parabolic type via bilinear control\tnoteref{fund}}
\tnotetext[fund]{This paper was partly supported by the INdAM National Group for Mathematical Analysis, Probability and their Applications.}



\author{Fatiha Alabau Boussouira}
\address{Laboratoire Jacques-Louis Lions Sorbonne Universit\'{e}, Universit\'{e} de Lorraine, 75005, Paris, France

alabau@ljll.math.upmc.fr}
\author{Piermarco Cannarsa\fnref{pc}} 
\address{Dipartimento di Matematica, Universit\`{a} di Roma Tor Vergata, 00133, Roma, Italy

cannarsa@mat.uniroma2.it}
\author{Cristina Urbani\corref{cor1}\fnref{cu}}
\address{Dipartimento di Matematica, Universit\`{a} di Roma Tor Vergata, 00133, Roma, Italy 

urbani@mat.uniroma2.it}
\cortext[cor1]{Corresponding Author}
\fntext[pc]{This author acknowledges support from the MIUR Excellence Department Project awarded to the Department of Mathematics, University of Rome Tor Vergata, CUP E83C18000100006.}
\fntext[cu]{This author is grateful to University Italo Francese (Vinci Project 2018).}

\begin{abstract}
In a separable Hilbert space $X$, we study the controlled evolution equation
\begin{equation*}
u'(t)+Au(t)+p(t)Bu(t)=0,
\end{equation*}
where $A\geq-\sigma I$ ($\sigma\geq0$) is a self-adjoint linear operator, $B$ is a bounded linear operator on $X$, and $p\in L^2_{loc}(0,+\infty)$ is a bilinear control.

We give sufficient conditions in order for the above nonlinear control system to be locally controllable to the $j$th eigensolution for any $j\geq1$. We also derive semi-global controllability results in large time and discuss applications to parabolic equations in low space dimension. Our method is constructive and all the constants involved in the main results can be explicitly computed.
\end{abstract}

\end{frontmatter}


\section{Introduction}
In a separable Hilbert space $X$ consider the nonlinear control system
\begin{equation}\label{u}
\left\{
\begin{array}{ll}
u'(t)+Au(t)+p(t)Bu(t)=0,& t>0\\\\
u(0)=u_0.
\end{array}\right.
\end{equation}
where $A:D(A)\subset X\to X$ is a linear self-adjoint operator on $X$ such that $A\geq-\sigma I$, with $\sigma\geq0$, $B$ belongs to $\mathcal{L}(X)$, the space of all bounded linear operators on $X$, and $p(t)$ is a scalar function representing a bilinear control. We suppose that the spectrum of $A$ consists of a sequence of real numbers $\{\lambda_k\}_{k\in\NN^*}$ which can be ordered, whithout loss of generality, as $-\sigma\leq\lambda_k\leq\lambda_{k+1}\to\infty$ as $k\to\infty$. We denote by $\{\varphi_k\}_{k\in\NN^*}$ the corresponding eigenfunctions, $A\varphi_k=\lambda_k\varphi_k,$ with $\norm{\varphi_k}=1$, $\forall\,k\in\NN^*$.

In the recent paper \cite{acu}, we studied the stabilizability of \eqref{u} to the $j$-th eigensolution of the free equation ($p\equiv0$), $\psi_j(t)=e^{-\lambda_j t}\varphi_j$, for every $j\in\NN^*$.
For this purpose, we introduced the notion of \emph{$j$-null controllability in time $T>0$} for the pair $\{A,B\}$:
denoting by $y(\cdot;v_0,p)$ the solution of the  linear system
\begin{equation*}
\left\{\begin{array}{ll}
y'(t)+Ay(t)+p(t)B\varphi_j=0,&t\in[0,T]\\\\
y(0)=y_0,
\end{array}\right.
\end{equation*}
we say that  $\{A,B\}$ is $j$-null controllable in time $T>0$ if  for any initial condition $y_0\in X$ there exists a control $p\in L^2(0,T)$  such that
\begin{equation*}
y(T;v_0,p)=0
\quad\mbox{and}\quad \norm{p}_{L^2(0,T)}\leq N_T\norm{y_0} ,
\end{equation*}
where  $N_T$ is a positive constant depending only on $T$.
Then, the \emph{control cost} is given by
\begin{equation*}
N(T)=\sup_{\norm{y_0}=1}\inf\left\{\norm{p}_{L^2(0,T)}\,:\, y(T;y_0,p)=0\right\}.\end{equation*}

In \cite[Theorem 3.7]{acu} we have shown that, if $\{A,B\}$ is $j$-null controllable, then \eqref{u} is locally superexponentially stabilizable to $\psi_j$: for all $u_0$ in some neighborhood of $\varphi_j$ there exists a control $p\in L^2_{loc}([0,+\infty))$ such that the corresponding solution $u$ of \eqref{u} satisfies 
\begin{equation}\label{superex}
\norm{u(t)-\psi_j(t)} \leq Me^{-e^{\omega t}},\qquad \forall\,t\geq0
\end{equation}
for suitable constants $\omega,M>0$ independent of $u_0$. Notice that such a result holds only under the condition of $j$-null controllability for the pair $\{A,B\}$. In particular, no assumptions are required on the behavior of the control cost.

Moreover, in \cite[Theorem 3.8]{acu} we gave sufficient conditions to ensure the $j$-null controllability of $\{A,B\}$: a gap condition for the eigenvalues of $A$ and a rank condition on $B$.

In this paper, we address the related, more delicate, issue of the exact controllability of \eqref{u} to the eigensolutions $\psi_j$ via bilinear controls. The main differences between the results of this paper and \cite[Theorem 3.7]{acu} can be summarized as follows:
\begin{itemize}
\item in addition to assuming the pair $\{A,B\}$ to be $j$-null controllable, we further require that the control cost $N(\cdot)$ satisfies $N(\tau)\leq e^{\nu/\tau}$ for any $0<\tau\leq T_0$, with $\nu,T_0>0$,
\item under the above stronger assumptions, not only we prove local exact controllability in any time, but also global exact controllablity in large time for a wide set of initial data.
\end{itemize}

The following result ensures local exact controllability for problem \eqref{u} assuming a precise behavior of the control cost for small time. In the last section of this paper, we show that such a behavior of the control cost is typical of parabolic problems in one space dimension.
\begin{thm}\label{teo1}
Let $A:D(A)\subset X\to X$ be a densely defined linear operator such that
\begin{equation}\label{ipA}
\begin{array}{ll}
(a) & A \mbox{ is self-adjoint},\\
(b) &\exists\,\sigma\geq0\,:\,\langle Ax,x\rangle \geq -\sigma\norm{x} ^2,\,\, \forall\, x\in D(A),\\
(c) &\exists\,\lambda>-\sigma \mbox{ such that }(\lambda I+A)^{-1}:X\to X \mbox{ is compact},
\end{array}
\end{equation}
and let $B:X\to X$ be a bounded linear operator. Assume that $\{A,B\}$ is $j$-null controllable in any time $T>0$ for some $j\in\NN^*$ and suppose that
\begin{equation}\label{bound-control-cost}
N(\tau)\leq e^{\nu/\tau},\quad\forall\,0<\tau\leq T_0,
\end{equation}
for some constants $\nu,T_0>0$.

Then, for any $T>0$, there exists a constant $R_{T}>0$ such that, for any $u_0\in B_{R_{T}}(\varphi_j)$, there exists a control $p\in L^2(0,T)$ such that the solution $u$ of \eqref{u} satisfies $u(T)=\psi_j(T)$.
Moreover, the following estimate holds
\begin{equation}\label{intro-estim-p}
\norm{p}_{L^2(0,T)}\leq \frac{e^{-\pi^2\Gamma_0/T}}{e^{2\pi^2\Gamma_0/(3T)}-1},
\end{equation}
where $\Gamma_0$ and $R_T$ can be computed as follows
\begin{equation}\label{Gamma_0}
\Gamma_0:=2\nu+\max\left\{\ln(D),0\right\},
\end{equation}
\begin{equation}\label{RT}
R_T:=e^{-6\Gamma_0/T_1},
\end{equation}
with
\begin{equation}\label{T_11}
T_1:=\min\left\{\frac{6}{\pi^2}T,1,T_0\right\},
\end{equation}
\begin{equation}\label{D}
D:=2\norm{B}e^{2\sigma+(3\norm{B})/2+1/2}\max\left\{1,\norm{B}\right\}.
\end{equation}
\end{thm}

The main idea of the proof consists of applying the stability estimates of \cite{acu} on a suitable sequence of time intervals of decreasing length $T_j$, such that $\sum_{j=1}^\infty T_j<\infty$. Such a sequence, which can be constructed only thanks to \eqref{bound-control-cost}, has to be carefully chosen in order to fit the error estimates that we take from \cite{acu}. We point out that our method is fully constructive, being based on an algorithm that allows to compute all relevant constants. In particular, we make no use of inverse mapping theorems.

In \cite{acu}, we gave sufficient conditions for $j$-null controllability. However, the hypotheses of \cite[Theorem 3.8]{acu} do not guarantee the validity of condition \eqref{bound-control-cost} for the control cost. In the result that follows, we provide sufficient conditions for $N(T)$ to satisfy \eqref{bound-control-cost}. It would be interesting to understand if \eqref{bound-control-cost} is also necessary  for the local exact controllability of \eqref{u}.

\begin{thm}\label{Thm-suff-cond}
Let $A:D(A)\subset X\to X$ be such that \eqref{ipA} holds and suppose that there exists a constant $\alpha>0$ for which the eigenvalues of $A$ fulfill the gap condition
\begin{equation}\label{gap}
\sqrt{\lambda_{k+1}-\lambda_1}-\sqrt{\lambda_k-\lambda_1}\geq \alpha,\quad\forall\, k\in \NN^*.
\end{equation}
Let $B: X\to X$ be a bounded linear operator such that there exist $b,q>0$ for which
\begin{equation}\label{ipB}
\begin{array}{l}
\langle B\varphi_j,\varphi_j\rangle\neq0\quad\mbox{and}\quad\left|\lambda_k-\lambda_1\right|^q|\langle B\varphi_j,\varphi_k\rangle|\geq b,\quad\forall\,k\neq j.
\end{array}
\end{equation}
Then, the pair $\{A,B\}$ is $j$-null controllable in any time $T>0$, and the control cost $N(T)$ satisfies \eqref{bound-control-cost} with
\begin{equation}\label{T_0}
T_0:=\min\left\{1,1/\alpha^2\right\},
\end{equation}
and $\nu=\Gamma_j$, where
\begin{eqnarray}\label{Gamma_jFatiha}
\lefteqn{2\Gamma_j(M,b,q,\alpha)}
\\
\nonumber
&=&M+ \frac{M^2}4 + (2q+5) e + \max\left\{ \ln\left(\dfrac{3M}{|\langle B\varphi_j,\varphi_j\rangle|^2}\right), \ln\left(\dfrac{3MC_q}{b^2}\right), \ln\left(\dfrac{3M C_{q,\alpha} }{b^2}\right),0\right\}
\end{eqnarray}
and 
\begin{equation}\label{M}
M:=C^2\left(1+\frac{1}{\alpha^2}\right)^2+2|\lambda_1|,
\end{equation}
\begin{equation}\label{Cq-Cqa}
C_q=2\left(\frac{2q}{e}\right)^{2q},\quad C_{q,\alpha}=\frac{2\Gamma(2q+1)}{\alpha\sqrt{\lambda_2-\lambda_1}}.
\end{equation}
Here $\Gamma(\cdot)$ is the Gamma function and $C$ is a positive constant independent  of $T$ and $\alpha$.
\end{thm}

Observe that assumption \eqref{ipB} is stronger that \cite[hypothesis (16)]{acu}. Nevertheless, it is satisfied by all the examples of parabolic problems that we presented in \cite{acu}.

From Theorems \ref{teo1} and \ref{Thm-suff-cond} we deduce the following Corollary.
\begin{cor}
Let $A:D(A)\subset X\to X$ be such that \eqref{ipA} holds and suppose that there exists a constant $\alpha>0$ for which \eqref{gap} is satisfied. Let $B: X\to X$ be a bounded linear operator that verifies \eqref{ipB} for some $b,q>0$. Then, problem \eqref{u} is locally controllable to the $j$th eigensolution $\psi_j$ in any time $T>0$.
\end{cor}
Furthermore, from Theorem \ref{teo1} we deduce two semi-global controllability results in the case of an accretive operator $A$. In the first one, Theorem \ref{teoglobal} below, we prove that all initial states lying in a suitable strip can be steered in finite time to the first eigensolution $\psi_1$ (see Figure~\ref{fig1}). Moreover, we give a uniform estimate for the controllability time depending on the size of the projection of the initial datum $u_0$ on $\varphi_1^\perp$.
\begin{thm}\label{teoglobal}
Let $A:D(A)\subset X\to X$ be a densely defined linear operator such that \eqref{ipA} holds with $\sigma=0$ and let $B:X\to X$ be a bounded linear operator. Let \{A,B\} be a $1$-null controllable pair which satisfies \eqref{bound-control-cost}. Then, there exists a constant $r_1>0$ such that for any $R>0$ there exists $T_{R}>0$ such that for all $u_0\in X$ with
\begin{equation}\label{ipu0}
\left|\langle u_0,\varphi_1\rangle-1\right|< r_1,\qquad
\norm{u_0-\langle u_0,\varphi_1\rangle\varphi_1}\leq R,
\end{equation}
problem \eqref{u} is exactly controllable to the first eigensolution $\psi_1(t)=e^{-\lambda_1 t}\varphi_1$ in time $T_{R}$.
\end{thm}
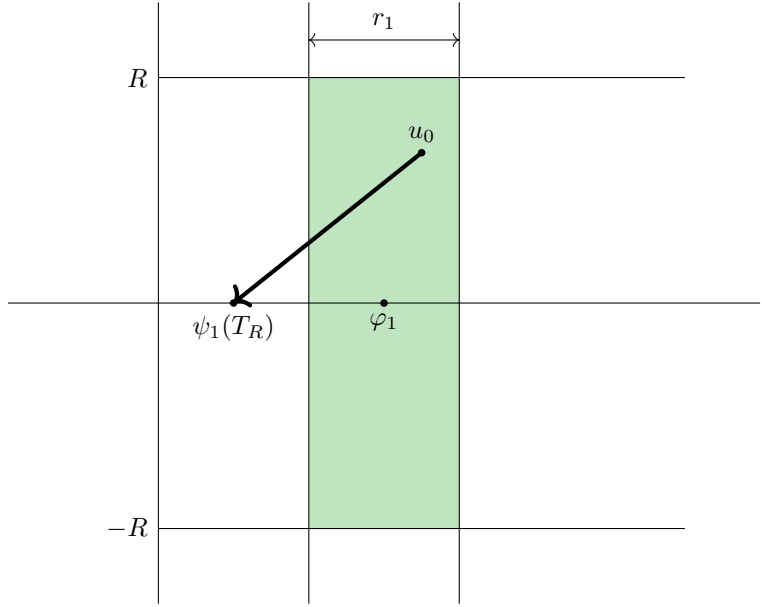
\begin{figure}[ht!]
\centering\begin{tikzpicture}
\fill[green(ryb)!30](4,-3)--(6,-3)--(6,3)--(4,3)--cycle;
\draw[] (0,0)  -- (10,0);
\draw[] (2,4) -- (2,-4);
\fill(5,0) node[below]{\footnotesize{$\varphi_1$}} circle (.05);
\fill(3,0) node[below]{\footnotesize{$\psi_1(T_R)$}} circle (.05);
\draw[] (4,-4) -- (4,4);
\draw[] (6,-4) -- (6,4);
\draw[] (2,3) node[left]{\footnotesize{$R$}} -- (9,3);
\draw[] (2,-3) node[left]{\footnotesize{$-R$}} -- (9,-3);
\draw[<->] (4,3.5) --node[above]{\footnotesize{$r_1$}} (6,3.5);
\fill(5.5,2) node[above]{\footnotesize{$u_0$}} circle (.05);
\draw[ultra thick,->] (5.5,2) -- (3,0);
\end{tikzpicture}
\caption{the colored region represents the set of initial conditions that can be steered to the first eigensolution in  time $T_R$.}\label{fig1}
\end{figure}

Our second semi-global result, Theorem \ref{teoglobal0} below, ensures the exact controllability of all initial states $u_0\in X\setminus \varphi_1^\perp$  to the evolution of their orthogonal projection along the first eigensolution. Such a function is defined by
\begin{equation}\label{exactphi1}
\phi_1(t)=\langle u_0,\varphi_1\rangle \psi_1(t), \quad\forall\, t \geq 0,
\end{equation}
where $\psi_1$ is the first eigensolution.

\begin{thm}\label{teoglobal0}
Let $A:D(A)\subset X\to X$ be a densely defined linear operator such that \eqref{ipA} holds with $\sigma=0$ and let $B:X\to X$ be a bounded linear operator. Let \{A,B\} be a 1-null controllable pair which satisfies \eqref{bound-control-cost}.
Then, for any $R>0$ there exists $T_R>0$ such that for all $u_0\in X$ with
\begin{equation}\label{cone}
\norm{u_0-\langle u_0,\varphi_1\rangle\varphi_1} \leq R |\langle u_0,\varphi_1\rangle|,
\end{equation}
system \eqref{u} is exactly controllable to $\phi_1$, defined in \eqref{exactphi1}, in time $T_R$.
\end{thm}
Notice that, denoting by $\theta$ the angle between the half-lines $\RR_+\varphi_1$ and $\RR_+ u_0$, condition \eqref{cone} is equivalent to
\begin{equation*}
|\tan\theta|\leq R,
\end{equation*}
which defines a closed cone, say $Q_R$, with vertex at $0$ and axis equal to $\RR\varphi_1$ (see Figure~\ref{fig2}). Therefore, Theorem \ref{teoglobal0} ensures a uniform controllability time for all initial conditions lying in $Q_R$.
We observe that, since $R$ is any arbitrary positive constant, all initial conditions $u_0\in X\setminus \varphi_1^\perp$ can be steered to the corresponding projection to the first eigensolution. Indeed, for any $u_0\in X\setminus \varphi_1^\perp$, we define 
\begin{equation*}
R_0:=\norm{\frac{u_0}{\langle u_0,\varphi_1\rangle}-\varphi_1}.
\end{equation*}
Then, for any $R\geq R_0$ condition \eqref{cone} is fulfilled:
\begin{equation*}
\frac{1}{|\langle u_0,\varphi_1\rangle|}\norm{u_0-\langle u_0,\varphi_1\rangle\varphi_1}=R_0\leq R.
\end{equation*}
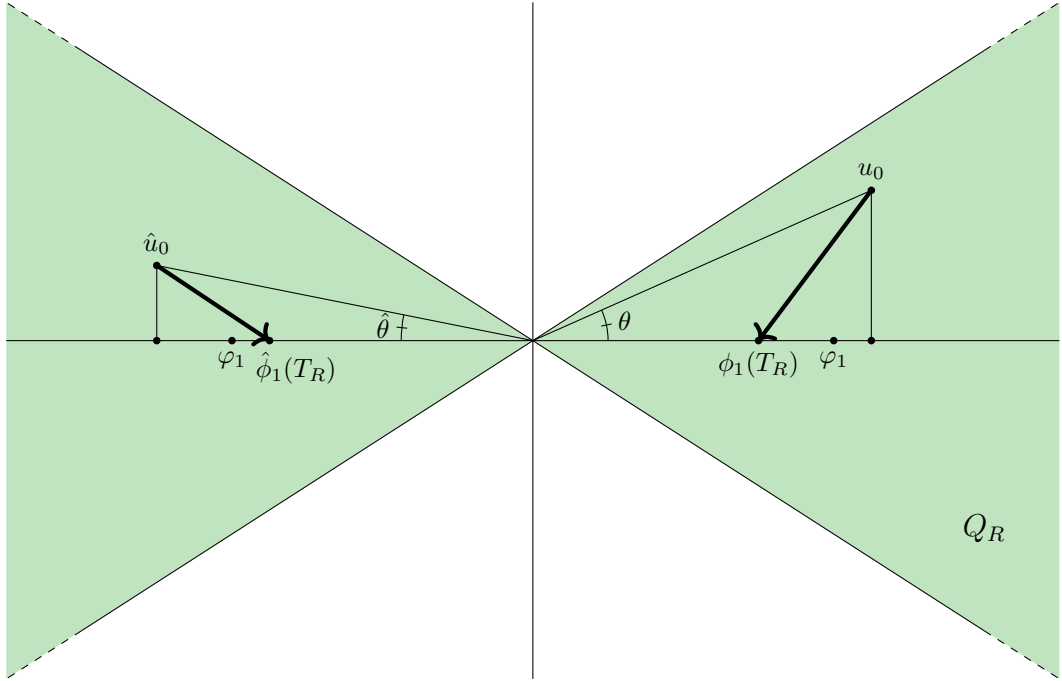
\begin{figure}[ht!]
\centering\begin{tikzpicture}
\fill[green(ryb)!30](0,-4.5)--(7,0)--(0,4.5)--cycle;
\fill[green(ryb)!30](7,0)--(14,-4.5)--(14,4.5)--cycle;
\coordinate (v2) at (7,0);
\coordinate(v4) at (11.5,2);
\coordinate(v5) at (11.5,0);
\coordinate(v8) at (2,1);
\coordinate(v9) at (2,0);
\tkzMarkAngle[size=1cm](v5,v2,v4);
\tkzLabelAngle[pos=1.25](v5,v2,v4){\footnotesize{$\theta$}};
\tkzMarkAngle[size=1.75cm](v8,v2,v9);
\tkzLabelAngle[pos=2](v8,v2,v9){\footnotesize{$\hat{\theta}$}};
\draw[] (0,0)  -- (14,0);
\draw[] (7,4.5) -- (7,-4.5);
\draw[] (7,0) -- (13,3.86);
\draw[dashed] (13,3.86) -- (14,4.5);
\draw[] (7,0) -- (13,-3.86);
\draw[dashed] (13,-3.86) -- (14,-4.5);
\draw[] (7,0) -- (1,3.86);
\draw[dashed] (1,3.86) -- (0,4.5);
\draw[] (7,0) -- (1,-3.86);
\draw[dashed] (1,-3.86) -- (0,-4.5);
\fill(11,0) node[below]{\footnotesize{$\varphi_1$}} circle (.05);
\fill(3,0) node[below]{\footnotesize{$\varphi_1$}} circle (.05);
\fill(10,0) node[below]{\footnotesize{$\phi_1(T_R)$}} circle (.05);
\fill(3.5,0) node[below]{\footnotesize{$\qquad\hat{\phi}_1(T_R)$}} circle (.05);
\fill(11.5,2) node[above]{\footnotesize{$u_0$}} circle (.05);
\fill(2,1) node[above]{\footnotesize{$\hat{u}_0$}} circle (.05);
\draw[] (11.5,2) -- (11.5,0);
\draw[] (11.5,2) -- (7,0);
\draw[] (2,1) -- (2,0);
\draw[] (2,1) -- (7,0);
\fill(11.5,0) circle (.05);
\fill(2,0) circle (.05);
\draw[ultra thick,->] (11.5,2) -- (10,0);
\draw[ultra thick,->] (2,1) -- (3.5,0);
\fill(13,-2.5) node{$Q_R$};
\end{tikzpicture}
\caption{fixed any $R>0$, the set of initial conditions exactly controllable in time $T_R$ to their projection along the first eigensolution is indicated by the colored cone $Q_R$.}\label{fig2}
\end{figure}

Finally, we would like to recall part of the huge literature on bilinear control of evolution equations, referring the reader to the references in \cite{acu} for more details. A seminal paper in this field is certainly the one by Ball, Marsden, Slemrod \cite{bms}, which establishes that system \eqref{u} is not controllable. More precisely, denoting by  $u(t;u_0,p)$  the unique solution of \eqref{u},  the attainable set from $u_0$ defined by
\begin{equation*}
S(u_0)=\{ u(t;u_0,p);t\geq 0, p\in L^r_{loc}([0,+\infty),\RR),r>1\}
\end{equation*}
is shown in \cite{bms} to have a dense complement.

As for positive results, we would like to mention Beauchard \cite{b}, on bilinear control of the wave equation, and Beauchard, Laurent \cite{bl} on  bilinear control of the Schr{\"o}dinger equation (see also \cite{beau} for a first result on this topic). The results obtained in these papers rely on linearization around the ground state, the use of the inverse mapping theorem, and a regularizing effect which takes place in both problems. Local  controllability is proved for any positive time for the Schr{\"o}dinger equation and for a sufficiently (optimal) large time for the wave equation. Both papers require the condition 
\begin{equation}
\label{B}
\langle B\varphi_1,\varphi_k\rangle \neq 0,\quad \forall \, k \geq 1
\end{equation}
to be satisfied, together with a suitable asymptotic behavior with respect to the eigenvalues. Notice that the structure of the second order operator and the fact that the  space dimension equals one allow the authors of  \cite{b} and \cite{bl} to apply Ingham's theory (\cite{kl}) which requires a gap condition on the eigenvalues. We further observe that even if the genericity of assumption \eqref{B} is proved in both papers \cite{b,bl}, only few explicit examples of operators $B$ of multiplication type are available in the literature. We refer to \cite{au} where a general constructive algorithm for building potentials which satisfy the infinite non-vanishing conditions \eqref{B}, and further the asymptotic condition \eqref{ipB}, is established.

If \eqref{B} is violated then it has been first shown by Coron \cite{cor}, for a model describing a particle in a moving box, that there exists a minimal time for local exact controllability to hold. This model couples the Schr\"odinger equation with two ordinary differential equations modeling the speed and acceleration of the box (see also Beauchard, Coron \cite{beaucor} for local exact controllability for large time). A further paper by Beauchard and Morancey \cite{bm} for the Schr\"odinger equation extends \cite{bl} to cases for which the above condition is violated, that is, when there exist integers $k$ such that 
$\langle B\varphi_1,\varphi_k\rangle =0$. 

An example of controllability to trajectories for nonlinear parabolic systems is studied in \cite{fgip}, where, however, additive controls are considered. In such an example, one can obtain controllability to free trajectories by Carleman estimates and inverse mapping arguments. Such a strategy seems hard to adapt to the current setting.

The paper which has the strongest connection with our work is the one by Beauchard and Marbach in \cite{beau-mar}, where the authors study small-time null controllability for  a scalar-input heat equation in one space dimension, with nonlinear lower order terms. Among the results of such paper, we mention  null-controllability to the first eigenstate of a heat equation with bilinear control. From this result it would be possible to deduce local  controllability only to the first eigenstate of the heat equation subject to Neumann boundary conditions. It is worth noting that \cite{beau-mar} addresses a specific parabolic equation. Moreover, the methods developed therein, relying on the so-called source term procedure, are totally different from ours.

We observe that the bilinear controls we use in this paper are just scalar functions of time. This fact explains why applications mainly concern problems in low space dimension, like the results in \cite{beau,b,beaucor,bl,beau-mar,bm,cor}. 
A stronger control action could be obtained  by letting controls depend on time and space. We refer the reader to \cite{cfk,ck} for more on this subject.

This paper is organized as follows. In section 2, we have collected some preliminaries as well as results from \cite{acu} that we need in order to prove Theorem \ref{teo1}. Section 3 contains such a proof, while section 4 is devoted to demonstrate Theorem \ref{Thm-suff-cond}. In section 5, we give the proof of our semi-global results (Theorems \ref{teoglobal} and \ref{teoglobal0}). Finally,  applications of Theorem \ref{teo1} to parabolic problems are analyzed is section 6.
\section{Preliminaries}
In this section, we recall a well-known result for the well-posedness of our control problem and the regularity of the solution as well as some results from \cite{acu} that are necessary for the proof of Theorem \ref{teo1}. Moreover, we will remind the fundamental definition of $j$-null controllable pair.

We recall our general functional frame. Let $(X,\langle\cdot,\cdot\rangle,\norm{\cdot})$ be a separable Hilbert space, let $A:D(A)\subset X\to X$ be a densely defined linear operator with the following properties
\begin{equation}\label{ipAA}
\begin{array}{ll}
(a) & A \mbox{ is self-adjoint},\\
(b) &\exists\,\sigma\geq0\,:\,\langle Ax,x\rangle \geq -\sigma\norm{x} ^2,\,\, \forall\, x\in D(A),\\
(c) &\exists\,\lambda>-\sigma \mbox{ such that }(\lambda I+A)^{-1}:X\to X \mbox{ is compact}.
\end{array}
\end{equation}
We denote by $\{\lambda_k\}_{k\in\NN^*}$ the eigenvalues of $A$, which can be ordered, whithout loss of generality, as $-\sigma\leq\lambda_k\leq\lambda_{k+1}\to\infty$ as $k\to\infty$, and by $\{\varphi_k\}_{k\in\NN^*}$ the corresponding eigenfunctions, $A\varphi_k=\lambda_k\varphi_k,$ with $\norm{\varphi_k}=1$, $\forall\,k\in\NN^*$.

Let $B:X\to X$ be a bounded linear operator. Fixed $T>0$, consider the following bilinear control problem
\begin{equation}\label{a1f}\left\{
\begin{array}{ll}
u'(t)+A u(t)+p(t)Bu(t)+f(t)=0,& t\in [0,T]\\\\
u(0)=u_0.
\end{array}\right.
\end{equation}
If $u_0\in X$, $p\in L^2(0,T)$ and $f\in L^2(0,T;X)$, a function $u\in C^0([0,T],X)$ is called a \emph{mild solution} of \eqref{a1f} if it satisfies
\begin{equation*}
u(t)=e^{-tA }u_0-\int_0^t e^{-(t-s)A}[p(s)Bu(s)+f(s)]ds, \quad\forall t\in[0,T].
\end{equation*}

We introduce the following notation: 
\begin{equation*}\begin{array}{l}
\norm{f}_{2}:=\norm{f}_{L^2(0,T;X)},\qquad\forall\,f\in L^2(0,T;X)\\\\
\norm{f}_{\infty}:=\norm{f}_{C([0,T];X)}=\sup_{t\in [0,T]}\norm{f(t)} ,\qquad\forall\, f\in C([0,T];X).
\end{array}
\end{equation*}
The well-posedness of \eqref{a1f} is ensured by the following proposition (see \cite{bms} for a proof).
\begin{prop}\label{propa24}
Let $T>0$. For any $u_0\in X$, $p\in L^2(0,T)$ and $f\in L^2(0,T;X)$ there exists a unique mild solution of \eqref{a1f}.

Furthermore, $u(\cdot)$ satisfies
\begin{equation}\label{a5}
\norm{u}_{\infty}\leq C(T) (\norm{u_0} +\norm{f}_{2}),
\end{equation}
for a suitable positive constant $C(T)$.
\end{prop}

\begin{oss}\label{rmk-regularity}
\emph{Under the hypotheses of Theorem \ref{propa24} it is possible to prove that the solution is more regular. Indeed, for every $\eps\in(0,T)$ it holds that $u\in H^1(\eps,T;X)\cap L^2(\eps,T;D(A))$ and the following identity is satisfied
\begin{equation*}
u'(t)+A u(t)+p(t)Bu(t)+f(t)=0,\quad \text{for a.e. }t\in[\eps,T].
\end{equation*}
Furthermore, if $u_0=0$ then $u\in H^1(0,T;X)\cap L^2(0,T;D(A))$ (it can be deduced by applying, for instance, \cite[Proposition 3.1, p.130]{bd}).}
\end{oss}

Let us now consider the following nonlinear control problem
\begin{equation}\label{v}
\left\{\begin{array}{ll}
v'(t)+A v(t)+p(t)Bv(t)+p(t)B\varphi_j=0,&t\in[0,T]\\\\
v(0)=v_0,
\end{array}\right.
\end{equation}
where $\varphi_j$ is the $j$th eigenfunction of $A$. We denote by $v(\cdot;v_0,p)$ the solution of \eqref{v} associated with initial condition $v_0$ and control $p$.

The following result establishes a bound for the solution of \eqref{v} in terms of the initial condition. We give its proof in \ref{appendix} for the sake of clarity and completeness. This proof follows that of \cite[Proposition 4.3]{acu}, with a different presentation, in particular with respect to the assumptions in the statement.
\begin{prop}\label{prop38}
Let $T>0$. Let $A:D(A)\subset X\to X$ be a densely defined linear operator that satisfies \eqref{ipAA} and let $B:X\to X$ be a bounded linear operator. Let $v_0\in X$ and let $p\in L^2(0,T)$ be such that 
\begin{equation}\label{prelim-p-bound}
\norm{p}_{L^2(0,T)}\leq N_T\norm{v_0},
\end{equation}
with $N_T$ a positive constant.

Then, $v(\cdot;v_0,p)$ verifies
\begin{equation}\label{unifv}
\sup_{t\in[0,T]}\norm{v(t;v_0,p)}^2\leq C_1(T,\norm{v_0})\norm{v_0}^2,
\end{equation}
where $C_1(T,\norm{v_0}):=e^{(2\sigma+\norm{B})T+2\norm{B}N_T\sqrt{T}\norm{v_0}}(1+\norm{B}N_T^2)$ and $\sigma$ is defined in \eqref{ipAA}.
\end{prop}
For any $0\leq s_0\leq s_1$, we now introduce the linear problem
\begin{equation}\label{newlin}
\begin{cases}
y'(t)+Ay(t)+p(t)B\varphi_j=0,&t\in[s_0,s_1]\\\\
y(s_0)=y_0
\end{cases}
\end{equation}
and we denote by $y(\cdot;y_0,s_0,p)$ the solution associated with initial condition $y_0$ at time $s_0$ and control $p$. Let us recall that for any fixed $T>0$ and $j\in\NN^*$, we say that the pair $\{A,B\}$ is \emph{$j$-null controllable in time $T$} if there exists a constant $N_T$ such that for every $y_0\in X$ there exists a control $p\in L^2(0,T)$ with
\begin{equation}\label{estimpnew}
\norm{p}_{L^2(0,T)}\leq N_T\norm{y_0},
\end{equation}
for which the solution of \eqref{newlin} with $s_0=0$ and $s_1=T$ satisfies $y(T;y_0,0,p)=0$. In this case, we define the \emph{control cost} as
\begin{equation}\label{def-control-cost}
N(T)=\sup_{\norm{y_0}=1}\inf\left\{\norm{p}_{L^2(0,T)}\,:\, y(T;y_0,0,p)=0\right\}.\end{equation}
With an approximation argument one realizes that \eqref{estimpnew} holds with $N_T=N(T)$, that is, for every $y_0\in X$ there exists $p\in L^2(0,T)$ with $\norm{p}_{L^2(0,T)}\leq N(T)\norm{y_0}$ such that $y(T;y_0,0,p)=0$.

Now, consider the following control problem
\begin{equation}\label{w}
\left\{\begin{array}{ll}
w'(t)+Aw(t)+p(t)Bv(t)=0,&t\in[0,T]\\\\
w(0)=0,
\end{array}\right.
\end{equation}
with $v$ the solution of \eqref{v}. We denote by $w(\cdot;0,p)$ the solution of \eqref{w} associated with control $p$.

In the following proposition we give a quadratic estimate of the solution of \eqref{w} in terms of the initial condition of the Cauchy problem solved by $v$. We give its proof in \ref{appendix} for the sake of clarity and completeness. This proof follows that of \cite[Proposition 4.4]{acu}, with a different presentation and a different hypothesis \eqref{v0} compared to the corresponding ones in the statement of \cite[Proposition 4.4]{acu}.
\begin{prop}\label{prop39}
Let $T>0$, $A:D(A)\subset X\to X$ be a densely defined linear operator that satisfies \eqref{ipAA} and $B:X\to X$ be a bounded linear operator. Let $p\in L^2(0,T)$ verify \eqref{prelim-p-bound} with $N_T=N(T)$ and $v_0\in X$ be such that 
\begin{equation}\label{v0}
N(T)\norm{v_0}\leq 1.
\end{equation}
Then, $w(\cdot;0,p)$ satisfies
\begin{equation}\label{wT}
\norm{w(T;0,p)}\leq K(T)\norm{v_0}^2,
\end{equation}
where
\begin{equation}\label{KT}
K^2(T):=\norm{B}^2N(T)^2e^{(4\sigma+\norm{B}+1)T+2\norm{B}\sqrt{T}}\left(1+\norm{B}N(T)^2\right).
\end{equation}
\end{prop}

\section{Proof of Theorem \ref{teo1}}
Fixed any $j\in\NN^*$ and any $T>0$, our aim is to prove local exact controllability in time $T$ for the following problem
\begin{equation}\label{sys}\left\{\begin{array}{ll}
u'(t)+A u(t)+p(t)Bu(t)=0,& t\in[0,T]\\\\
u(0)=u_0,
\end{array}\right.
\end{equation}
to the $j$th eigensolution $\psi_j=e^{-\lambda_j t}\varphi_j$ of $A$, that is the solution of \eqref{sys} when $p=0$ and $u_0=\varphi_j$. Hereafter, we will denote by $u(\cdot;u_0,p)$ the solution of \eqref{sys} associated with initial condition $u_0$ and control $p$. 

We recall that $A:D(A)\subset X\to X$ is a densely defined linear operator that satisfies
\begin{equation}\label{ipAAA}
\begin{array}{ll}
(a) & A \mbox{ is self-adjoint},\\
(b) &\exists\,\sigma\geq0\,:\,\langle Ax,x\rangle \geq -\sigma\norm{x} ^2,\,\, \forall\, x\in D(A),\\
(c) &\exists\,\lambda>-\sigma \mbox{ such that }(\lambda I+A)^{-1}:X\to X \mbox{ is compact},
\end{array}
\end{equation}
and we denote by $\{\lambda_k\}_{k\in\NN^*}$ and $\{\varphi_k\}_{k\in\NN^*}$ the eigenvalues and the eigenfunctions of $A$, respectively. $B:X\to X$ is a bounded linear operator. The pair $\{A,B\}$ is assumed to be $j$-null controllable in any time, with control cost that satisfies
\begin{equation}\label{bound-control-costt}
N(\tau)\leq e^{\nu/\tau},\quad\forall\,0<\tau\leq T_0,
\end{equation}
for some constants $\nu,T_0>0$.

The proof of Theorem \ref{teo1} is divided into two main parts: the case $\lambda_j=0$, that we build by a series of steps, and the case $\lambda_j\neq0$. 

\subsection{Case $\lambda_j=0$}
If $\lambda_j=0$ our reference trajectory will be the stationary function $\psi_j\equiv\varphi_j$. Given $T>0$, we define $T_f$ as
\begin{equation}\label{T_f}
T_f:=\min\left\{T,\frac{\pi^2}{6},\frac{\pi^2}{6}T_0\right\},
\end{equation}
where $T_0$ is the constant in \eqref{bound-control-costt}. We will actually build a control $p\in L^2(0,T_f)$ such that $u(T_f;u_0,p)=\psi_j$, and then, by taking $p(t)\equiv 0$ for $t>T_f$, the solution $u$ of \eqref{sys} will remain forever on the target trajectory $\psi_j$.

Now, we define
\begin{equation}\label{T_1}
T_1:=\frac{6}{\pi^2}T_f,
\end{equation}
and we observe that $0<T_1\leq 1$. Then, we introduce the sequence $\{T_j\}_{j\in\NN^*}$ as
\begin{equation}\label{Tj}
T_j:=T_1/j^2,
\end{equation}
and the time steps
\begin{equation}\label{taun}
\tau_n=\sum_{j=1}^n T_j,\qquad\forall\, n\in\NN,
\end{equation}
with the convention that $\sum_{j=1}^0T_j=0$. Notice that $\sum_{j=1}^\infty T_j=\frac{\pi^2}{6}T_1=T_f$.
\begin{oss}\label{ossF32}
Note that the sequence of times $(T_j)_{j \in NN^*}$ is strictly decaying towards $0$, whereas the sequence of times
 $(\tau_j)_{j \in NN^*}$ is strictly increasing and converges to $T_f$.
 \end{oss}
 
Set $v:=u-\varphi_j$. We will consider the equation satisfied by $v$ on suitable intervals of time $[s_0,s_1]$ and suitable initial data $v^0$ at the initial time $s_0$, as follows.  Given any $0\leq s_0\leq s_1\leq T$, and any $v^0$ in $X$, $v$ is the solution of the following Cauchy problem
\begin{equation}\label{vv}
\left\{\begin{array}{ll}
v'(t)+A v(t)+p(t)Bv(t)+p(t)B\varphi_j=0,&t\in[s_0,s_1]\\\\
v(s_0)=v^0.
\end{array}\right.
\end{equation}
We denote by $v(\cdot;v^0,s_0,p)$ the solution of \eqref{vv} associated with initial condition $v^0$ at time $s_0$ and control $p$. Observe that proving the controllability of $u$ to $\psi_j=\varphi_j$ in time $T_f$ is equivalent to show the null controllability of $v$, that is, $v(T_f;v_0,0,p)=0$, where $v_0=u_0-\varphi_j$.

The strategy of the proof consists first of building a control $p_1\in L^2(0,T_1)$ such that at time $T_1$ the solution of \eqref{vv} can be estimated by the square of the initial condition. We then iterate the procedure on consecutive time intervals of the form $[\tau_{n-1},\tau_n]$: each time we construct a control $p_n\in L^2(\tau_{n-1},\tau_n)$ such that the solution of \eqref{vv} on $[\tau_{n-1},\tau_n]$ at time $\tau_n$ is estimated by the square of the initial condition on such interval. Hence, combining all those estimates and letting $n$ go to infinity, we finally deduce that there exists a control $p\in L^2_{loc}(0,+\infty)$ such that $v(T_f;v_0,0,p)=0$ and so $u(T_f;u_0,p)=\varphi_j$.

In practice, we shall build, by induction, controls $p_n\in L^2(\tau_{n-1},\tau_n)$ for $n\geq 1$ such that, setting 
\begin{equation}\label{def-q-v}
\begin{array}{l}
\displaystyle q_{n}(t):=\sum_{j=1}^n p_j(t)\chi_{[\tau_{j-1},\tau_j]}(t),\\
v_n:=v(\tau_n;v_0,0,q_n),
\end{array}
\end{equation}
it holds that
\begin{equation}\label{iterative-stepp}
\begin{array}{ll}
1.&\norm{p_{n}}_{L^2(\tau_{n-1},\tau_n)}\leq N(T_n)\norm{v_{n-1}},\\
2.& y(\tau_{n};v_{n-1},\tau_{n-1},p_n)=0,\\
3.&\norm{v(\tau_n;v_{n-1},\tau_{n-1},p_n)}\leq  e^{\left(\sum_{j=1}^n 2^{n-j}j^2-2^n6\right)\Gamma_0/T_1},\\
4.&\norm{v(\tau_n;v_{n-1},\tau_{n-1},p_n)}\leq \prod_{j=1}^{n}K(T_j)^{2^{n-j}}\norm{v_0}^{2^{n}}.
\end{array}
\end{equation}
Observe that, by construction, 
\begin{equation*}
v_{n}=v(\tau_n;v_0,0,q_n)=v(\tau_n;v_{n-1},\tau_{n-1},p_n),\quad\forall\,n\geq1.
\end{equation*}

\subsubsection{First iteration}
Let us start by studying control problem \eqref{vv} in the first time interval $[s_0,s_1]=[\tau_0,\tau_1]=[0,T_1]$. Recalling that $\{A,B\}$ is $j$-null controllable in any time, given $v_0\in X$ there exists a control $p_1\in L^2(0,T_1)$ such that
\begin{equation}\label{p_0}
\norm{p_1}_{L^2(0,T_1)}\leq N(T_1)\norm{v_0},\quad\text{and}\quad y(T_1;v_0,0,p_1)=0,
\end{equation}
where $N(T_1)$ is the control cost and $y(\cdot;v_0,0,p_1)$ is the solution of the linear problem \eqref{newlin}. So, the first two items of \eqref{iterative-stepp} for $n=1$ are fulfilled. We now apply Proposition \ref{prop38} deducing that
\begin{equation}
\sup_{t\in[0,T_1]}\norm{v(t;v_0,0,p_1)}^2\leq C_1(T_1,\norm{v_0})\norm{v_0}^2,
\end{equation}
where $C_1(T_1,\norm{v_0})=e^{(2\sigma+\norm{B})T_1+2\norm{B}N(T_1)\sqrt{T_1}\norm{v_0}}(1+\norm{B}N(T_1)^2)$.

We measure how close from $0$ the solution of \eqref{vv} is steered at time $T_1$ by control $p_1$. For this purpose, we introduce the function $w(\cdot):=v(\cdot;v_0,0,p_1)-y(\cdot;v_0,0,p_1)$ which satisfies the following Cauchy problem
\begin{equation}\label{ww}
\left\{\begin{array}{ll}
w'(t)+Aw(t)+p_1(t)Bv(t)=0,&t\in[0,T_1]\\\\
w(0)=0.
\end{array}\right.
\end{equation}
Thanks to Proposition \ref{prop39}, if
\begin{equation}\label{v00}
N(T_1)\norm{v_0}\leq 1,
\end{equation}
then, the solution of \eqref{ww} satisfies
\begin{equation}\label{wTT}
\norm{w(T_1;0,p_1)}=\norm{v(T_1;v_0,0,p_1)}\leq K(T_1)\norm{v_0}^2,
\end{equation}
where $K(\cdot)$ is defined on $(0,\infty)$ as
\begin{equation}\label{KTT}
K^2(\tau):=\norm{B}^2N(\tau)^2e^{(4\sigma+\norm{B}+1)\tau+2\norm{B}\sqrt{\tau}}\left(1+\norm{B}N(\tau)^2\right).
\end{equation}
Notice that, the first equality in \eqref{wTT} holds true because control $p_1$ steers to $0$ the solution of the linear problem (see \eqref{p_0}). 

\begin{oss}\label{oss32}
Observe that function $K(\cdot)$ satisfies
$$K^2(\tau)\leq \norm{B}^2N^2(\tau)e^{(4\sigma+3\norm{B}+1)}\left(1+\norm{B}N^2(\tau)\right),\quad \forall\,0<\tau\leq1.$$
Therefore, since $T_1=\min\{6T/\pi^2, 1,T_0\}$, combining the above inequality with \eqref{bound-control-costt}, we deduce that there exists a constant $\Gamma_0>\nu$ such that
\begin{equation}\label{estK}
K(\tau)\leq e^{\Gamma_0/\tau},\quad \forall\,0<\tau\leq T_1.
\end{equation}
where $T_0$ is defined in \eqref{bound-control-costt}. 

Note that a suitable choice of constant $\Gamma_0$ such that \eqref{estK} holds is \eqref{Gamma_0}.
\end{oss}
We now define the radius of the neighborhood of $\varphi_j$ where we take the initial condition $u_0$ as in \eqref{RT}. Let $u_0\in B_{R_T}(\varphi_j)$, or equivalently $v_0=u_0 - \varphi_j \in B_{R_T}(0)$, be chosen arbitrarily. With this choice we have that
\begin{equation*}
N(T_1)\norm{v_0}\leq e^{\nu/T_1}e^{-6\Gamma_0/T_1}\leq e^{-5\Gamma_0/T_1}\leq1,
\end{equation*}
and \eqref{v00} is satisfied. Therefore, we get that
\begin{equation}\label{first-iteration-induction-for-v_n}
\norm{v(T_1;v_0,0,p_1)}\leq K(T_1)\norm{v_0}^2\leq e^{-11\Gamma_0/T_1},
\end{equation}
which proves 3. and 4. of \eqref{iterative-stepp} for $n=1$.

\subsubsection{Iterative step}
Now, suppose that we have built controls $p_j\in L^2(\tau_{j-1},\tau_j)$ such that \eqref{iterative-stepp} holds for each $j=1,\dots,n-1$. In particular, for $j=n-1$, there exists $p_{n-1}\in L^2(\tau_{n-1},\tau_n)$ which verifies
\begin{equation}\label{inductive-hyp-iter-step}
\begin{array}{ll}
1.&\norm{p_{n-1}}_{L^2(\tau_{n-2},\tau_{n-1})}\leq N(T_{n-1})\norm{v_{n-2}},\\
2.& y(\tau_{n-1};v_{n-2},\tau_{n-2},p_{n-1})=0,\\
3.&\norm{v(\tau_{n-1};v_{n-2},\tau_{n-2},p_{n-1})}\leq  e^{\left(\sum_{j=1}^{n-1} 2^{n-1-j}j^2-2^{n-1}6\right)\Gamma_0/T_1},\\
4.&\norm{v(\tau_{n-1};v_{n-2},\tau_{n-2},p_{n-1})}\leq \prod_{j=1}^{n-1}K(T_j)^{2^{n-1-j}}\norm{v_0}^{2^{n-1}}.
\end{array}
\end{equation}
We shall now prove that there exists $p_n\in L^2(\tau_{n-1},\tau_n)$ such that every item of \eqref{iterative-stepp} is fulfilled. We defined $q_{n-1}$ and $v_{n-1}$ as in \eqref{def-q-v} and we consider the following problem 
\begin{equation}\label{v-iterative-step}
\begin{cases}
v'(t)+Av(t)+p(t)Bv(t)+p(t)B\varphi_j=0,&[\tau_{n-1},\tau_n]\\
v(\tau_{n-1})=v_{n-1},
\end{cases}
\end{equation}
where the control $p$ has still to be suitably chosen. By the change of variables $s=t-\tau_{n-1}$ and the definition \eqref{taun}, we shift the problem from $[\tau_{n-1},\tau_n]$ into the interval $[0,T_n]$. We introduce the functions $\tilde{v}(s)=v(s+\tau_{n-1})$ and $\tilde{p}(s)=p\left(s+\tau_{n-1}\right)$ and we rewrite \eqref{v-iterative-step} as
\begin{equation}\label{tildevtn-1tn}
\left\{\begin{array}{ll}
\tilde{v}'(s)+A\tilde{v}(s)+\tilde{p}(s)B\tilde{v}(s)+\tilde{p}(s)B\varphi_j=0,&s\in \left[0,T_n\right]\\\\
\tilde{v}(0)=v_{n-1}.
\end{array}
\right.
\end{equation}
Recalling that $\{A,B\}$ is $j$-null controllable in any time, there exists a control $\tilde{p}_n\in L^2(0,T_n)$ such that
\begin{equation*}
\norm{\tilde{p}_n}_{L^2(0,T_n)}\leq N(T_n)\norm{v_{n-1}}\quad\text{and}\quad \tilde{y}(T_n,v_{n-1},0,\tilde{p}_n)=0,
\end{equation*}
where $\tilde{y}(\cdot;v_{n-1},0,\tilde{p}_n)$ is the solution of the linear problem \eqref{newlin} on $[0,T_n]$. Furthermore, since $v_{n-1}=v(\tau_{n-1};v_0,0,q_{n-1})=v(\tau_{n-1};v_{n-2},\tau_{n-2},p_{n-1})$, from 3. of \eqref{inductive-hyp-iter-step} we obtain that
\begin{equation}\label{induc}
\begin{split}
N(T_n)\norm{v_{n-1}}&\leq e^{\nu n^2/T_1}e^{\left(\sum_{j=1}^{n-1} 2^{n-1-j}j^2-2^{n-1}6\right)\Gamma_0/T_1}\\
&\leq e^{(n^2+(-(n-1)^2-4(n-1)+2^{n-1}6-6-2^{n-1}6)\Gamma_0/T_1}\\
&=e^{-(2n+3)\Gamma_0/T_1}\leq1,
\end{split}
\end{equation}
where we have used that the constant of the control cost $\nu$ is less than or equal to $\Gamma_0$ (see Remark \ref{oss32}), and the identity
\begin{equation}
\sum_{j=0}^n\frac{j^2}{2^j}=2^{-n}(-n^2-4n+6(2^n-1)), \qquad n\geq0,
\end{equation}
which can be easily checked by induction.

We now choose the control $\tilde{p}=\tilde{p}_n$ in \eqref{tildevtn-1tn} and still denote by $\tilde{v}$ the corresponding solution. We set $w=\tilde{v}-\tilde{y}$. Then, $w$ solves \eqref{w} with $T=T_n$ and $p=\tilde{p}_n$. So, we can apply Proposition \ref{prop39} with $T=T_n$ to problem \eqref{tildevtn-1tn} and since $w(T_n;0,\tilde{p}_n)=\tilde{v}(T_n;v_{n-1},0,\tilde{p}_n)$, we obtain that
\begin{equation*}
\norm{\tilde{v}(T_n;v_{n-1},0,\tilde{p}_n)}\leq K(T_n)\norm{v_{n-1}}^2.
\end{equation*}
We shift back the problem into the original interval $\left[\tau_{n-1},\tau_{n}\right]$, we define $p_n(t):=\tilde{p}_n( t-\tau_{n-1})$, and we get
\begin{equation*}
\norm{p_n}_{L^2(\tau_{n-1},\tau_n)}\leq N(T_n)\norm{v_{n-1}},\quad\text{and}\quad y(\tau_n,v_{n-1},\tau_{n-1},p_n)=0,
\end{equation*}
and
\begin{equation}\label{iterative-step-bound-v-taun}
\norm{v(\tau_{n};v_{n-1},\tau_{n-1},p_n)}\leq K(T_n)\norm{v_{n-1}}^2.
\end{equation}
So, we have proved the first two items of \eqref{iterative-stepp}. Moreover, thanks to 3. of \eqref{inductive-hyp-iter-step}, we deduce that
\begin{equation}
\norm{v(\tau_{n};v_{n-1},\tau_{n-1},p_n)}\leq e^{\Gamma_0 n^2/T_1}\left[e^{\left(\sum_{j=1}^{n-1} 2^{n-1-j}j^2-2^{n-1}6\right)\Gamma_0/T_1}\right]^2=e^{\left(\sum_{j=1}^n 2^{n-j}j^2-2^n6\right)\Gamma_0/T_1},
\end{equation}
that is the third item of \eqref{iterative-stepp}. Finally, using again \eqref{iterative-step-bound-v-taun} and 4. of \eqref{inductive-hyp-iter-step} we obtain that
\begin{equation*}
\norm{v(\tau_{n};v_{n-1},\tau_{n-1},p_n)}\leq K(T_n)\left[\prod_{j=1}^{n-1}K(T_j)^{2^{n-1-j}}\norm{v_0}^{2^{n-1}}\right]^2=\prod_{j=1}^{n}K(T_j)^{2^{n-j}}\norm{v_0}^{2^{n}}.
\end{equation*}
This concludes the induction argument and the proof of \eqref{iterative-stepp}.

We are now ready to complete the proof of Theorem \ref{teo1} for the case $\lambda_j=0$. We observe that for all $n\in\NN^*$
\begin{equation}\label{final}
\begin{split}
\norm{v(\tau_{n};v_{n-1},\tau_{n-1},p_n)}&\leq\prod_{j=1}^nK(T_j)^{2^{n-j}}\norm{v_0}^{2^n}\\
&\leq \prod_{j=1}^n\left(e^{\Gamma_0 j^2/T_1}\right)^{2^{n-j}}\norm{v_0}^{2^n}=e^{\Gamma_0 2^n/T_1\sum_{j=1}^nj^2/2^j}\norm{v_0}^{2^n}\\
&\leq e^{\Gamma_0 2^n/T_1\sum_{j=1}^\infty j^2/2^j}\norm{v_0}^{2^n}\leq \left(e^{6\Gamma_0/T_1}\norm{v_0}\right)^{2^n}
\end{split}
\end{equation}
where we have used that $\sum_{j=1}^\infty j^2/2^j=6$. Notice that \eqref{final} is equivalent to
\begin{equation}\label{final2}
\norm{v(\tau_{n};v_0,0,q_n)}\leq \left(e^{6\Gamma_0/T_1}\norm{v_0}\right)^{2^n},
\end{equation}
where $q_n(t)=\sum_{j=1}^{n}p_j(t)\chi_{[\tau_{j-1},\tau_j]}(t)$. We now take the limit as $n\to \infty$ in \eqref{final2} and we get
\begin{equation}
\norm{u\left(\frac{\pi^2}{6}T_1;u_0,q_{\infty}\right)-\varphi_j}=\norm{v\left(\frac{\pi^2}{6}T_1;v_0,0,q_{\infty}\right)}=\norm{v(T_f;v_0,0,q_{\infty})}\leq 0
\end{equation}
because $\norm{v_0}<e^{-6\Gamma_0/T_1}$. This means that, we have built a control $p\in L^2_{loc}([0,\infty))$, defined by
\begin{equation}
p(t)=\left\{\begin{array}{ll}
\sum_{n=1}^\infty p_{n}(t)\chi_{\left[\tau_{n-1} ,\tau_{n}\right]}(t),& t\in \left(0,T_f\right]\\\\
0,&t\in(T_f,+\infty)
\end{array}\right.
\end{equation}
for which the solution $u$ of \eqref{sys} reaches the $j$th eigensolution $\psi_j=\varphi_j$ in time $T_f$, less than or equal to $T$, and stays on it forever.

Observe that, thanks to the first item of \eqref{iterative-stepp} and to \eqref{induc}, we are able to yield a bound for the $L^2$-norm of such a control:
\begin{equation}\label{pestimate}
\begin{split}
\norm{p}^2_{L^2\left(0,T\right)}&=\sum_{n=1}^\infty \norm{p_{n}}^2_{L^2\left(\tau_{n-1},\tau_{n}\right)}\leq \sum_{n=1}^\infty \left(N(T_n)\norm{v\left(\tau_{n-1} \right)}\right)^2\leq \sum_{n=1}^\infty e^{-2(2n+3)C_K/T_1}\\
&\leq\frac{e^{-6C_K/T_1}}{e^{4C_K/T_1}-1}=\frac{e^{-\pi^2C_K/T_f}}{e^{2\pi^2C_K/(3T_f)}-1}.
\end{split}
\end{equation}
Notice that since \eqref{T_f} holds, \eqref{pestimate} implies \eqref{intro-estim-p}.

\subsection{Case $\lambda_j\neq0$}
Now, we face the case $\lambda_j\neq0$. We define the operator
\begin{equation*}
A_j:=A-\lambda_jI.
\end{equation*}
We proved in \cite[Lemma 4.7]{acu} that if $\{A,B\}$ is $j$-null controllable, then the same holds for the pair $\{A_j,B\}$. Furthermore, it is easy to check that also condition \eqref{bound-control-costt} is verified by the control cost associated with $\{A_j,B\}$, if the same property holds for the control cost associated with the pair $\{A,B\}$.

It is possible to check that $A_j$ satisfies \eqref{ipAAA} and moreover it has the same eigenfuctions, $\{\varphi_k\}_{k\in\NN^*}$, of $A$, while the eigenvalues are given by
\begin{equation*}
\mu_k=\lambda_k-\lambda_j, \qquad\forall\, k\in\NN^*.
\end{equation*}
In particular, $\mu_j=0$.

We define the function $z(t)=e^{\lambda_j t}u(t)$, where $u$ is the solution of \eqref{sys}. Then, $z$ solves the following problem
\begin{equation}\label{z}
\left\{\begin{array}{ll}
z'(t)+A_jz(t)+p(t)Bz(t)=0,&[0,T],\\\\
z(0)=u_0.
\end{array}\right.
\end{equation}
We define $T_f$ as in \eqref{T_f} (where $T_0$ is now the constant associated with the control cost relative to the pair $\{A_j,B\}$) and $R_{T}$ as in \eqref{RT}. We deduce from the previous analysis that, if $u_0\in B_{R_{T}}(\varphi_j)$, then there exists a control $p\in L^2([0,+\infty))$ that steers the solution $z$ to the eigenstate $\varphi_j$ in time $T_f\leq T$. This implies the exact controllability of $u$ to the eigensolution $\psi_j(t)=e^{-\lambda_j t}\varphi_j$: indeed,
\begin{equation*}
\norm{u\left(T_f;u_0,p\right)-\psi_j\left(T_f\right)}=\norm{e^{-\lambda_jT_f}z\left(T_f\right)-e^{-\lambda_jT_f}\varphi_j}=e^{-\lambda_jT_f}\norm{z\left(T_f\right)-\varphi_j}=0.
\end{equation*}

\begin{oss}
We observe that, from \eqref{pestimate}, it follows that $\norm{p}_{L^2(0,T_{f})}\to 0$ as $T_f\to 0$. This fact is not surprising since as $T_f$ approaches $0$, also the size of the neighborhood where the initial condition can be chosen goes to zero.
\end{oss}

\section{Proof of Theorem \ref{Thm-suff-cond}}
Before showing the proof of Theorem \ref{Thm-suff-cond}, we define formally the following function
\begin{equation}\label{G}
G_M(T):=\frac{M}{T^4}e^{M/T}\sum_{k=1}^{\infty}\frac{e^{-2\omega_kT+M\sqrt{\omega_k}}}{|\langle B\varphi_j,\varphi_k\rangle|^2},
\end{equation}
where $M$ is a positive constant, $\omega_k:=\lambda_k-\lambda_1$, for all $k\in\NN^*$, $\{\lambda_k\}_{k\in\NN^*}$ are the eigenvalues of $A$. In Lemma \ref{lem1} below, we investigate the behavior of $G_M(T)$ for small values of $T$. Such a result will be crucial for the analysis of the control cost $N(T)$ in Theorem \ref{Thm-suff-cond}.

\begin{lem}\label{lem1}
Let $A:D(A)\subset X\to X$ be such that \eqref{ipA} and \eqref{gap} hold and $B:X\to X$ be such that \eqref{ipB} holds. Then, for any $M,T>0$ the series in \eqref{G} is convergent and there exists a positive constant $\Gamma_j$, such that
\begin{equation} 
G_M(T)\leq e^{2\Gamma_j/T},\quad\forall\,0<T\leq 1.
\end{equation}
Moreover, a suitable choice of $\Gamma_j=\Gamma_j(M,b,q, \alpha)$ is \eqref{Gamma_jFatiha}.
\end{lem}
\begin{proof}
Thanks to assumption \eqref{ipB}, we have that
\begin{equation}
\begin{split}
G_M(T)&=\frac{M}{T^4}e^{M/T}\sum_{k=1}^{\infty}\frac{e^{-2\omega_kT+M\sqrt{\omega_k}}}{|\langle B\varphi_j,\varphi_k\rangle|^2}\\
&\leq \frac{M}{T^4}e^{M/T}\left[\frac{e^{M^2/(8T)}}{|\langle B\varphi_j,\varphi_j\rangle|^2}+\frac{1}{b^2}\sum_{k\neq j,\,k=1}^{\infty}\left(\omega_k^{2q}e^{-\omega_kT}\right)e^{-\omega_kT+M\sqrt{\omega_k}}\right].
\end{split}
\end{equation}
For any $\omega\geq0$ we set $f(\omega)=e^{-\omega T+M\sqrt{\omega}}$. The maximum value of $f$ is attained at $\omega=\left(\frac{M}{2T}\right)^2$. So, we can bound $G_M(T)$ as follows
\begin{equation}\label{gamma4}
G_M(T)\leq\frac{M}{T^4}e^{M/T}\left[\frac{e^{M^2/(8T)}}{|\langle B\varphi_j,\varphi_j\rangle|^2}+\frac{e^{M^2/(4T)}}{b^2}\sum_{k=1}^{\infty}\omega_k^{2q}e^{-\omega_kT}\right].
\end{equation}

Now, for any $\omega\geq0$ we define the function $g(\omega)=\omega^{2q}e^{-\omega T}$. Its derivative is given by
\begin{equation*}
g'(\omega)=(2q-\omega T)\omega^{2q-1}e^{-\omega T}
\end{equation*}
and therefore we deduce that
\begin{equation*}
g(\omega)\mbox{ is }\left\{\begin{array}{ll}\mbox{increasing } & \mbox{if }0\leq \omega<(2q)/T\\\\
\mbox{decreasing}&\mbox{if } \omega\geq (2q)/T  \end{array}\right.
\end{equation*}
and $g$ has a maximum at $\omega=(2q)/T$. We define the following index:
\begin{equation*}
k_1:=k_1(T)=\sup\left\{k\in\NN^*\,:\,\omega_k\leq \frac{2q}{T}\right\}
\end{equation*}
Note that $k_1(T)$ goes to $\infty$ as $T$ converges to $0$.
We can rewrite the sum in \eqref{gamma4} as follows
\begin{equation}\label{s123}
\sum_{k=1}^{\infty}\omega_k^{2q}e^{-\omega_kT}=\sum_{k\leq k_1-1}\omega_k^{2q}e^{-\omega_kT}+\sum_{k_1\leq k\leq k_1+1}\omega_k^{2q}e^{-\omega_kT}+\sum_{k\geq k_1+2}\omega_k^{2q}e^{-\omega_kT}.
\end{equation}

For any $k\leq k_1-1$, we have
\begin{equation}\label{s1}
\int_{\omega_k}^{\omega_{k+1}}\omega^{2q}e^{-\omega T}d\omega\geq (\omega_{k+1}-\omega_k)\omega_k^{2q}e^{-\omega_k T}\geq \alpha\sqrt{\omega_2}\,\omega_k^{2q}e^{-\omega_k T}
\end{equation}
and for any $k\geq k_1+2$
\begin{equation}\label{s3}
\int_{\omega_{k-1}}^{\omega_k}\omega^{2q}e^{-\omega T}d\omega\geq (\omega_k-\omega_{k-1})\omega_k^{2q}e^{-\omega_k T}\geq \alpha\sqrt{\omega_2}\,\omega_k^{2q}e^{-\omega_k T}.
\end{equation}
So, by using estimates \eqref{s1} and \eqref{s3}, \eqref{s123} becomes
\begin{equation}\label{s2i}
\sum_{k=1}^{\infty}\omega_k^{2q}e^{-\omega_kT}\leq\frac{2}{\alpha\sqrt{\omega_2}}\int_0^\infty \omega^{2q}e^{-\omega T}d\omega+\sum_{k_1\leq k\leq k_1+1}\omega_k^{2q}e^{-\omega_kT}.
\end{equation}
Furthermore, recalling that $g$ has a maximum for $\omega=2q/T$, it holds that
\begin{equation}\label{s2}
k=k_1,k_1+1\quad\Rightarrow\quad \omega_k^{2q}e^{-\omega_k T}\leq\left(2q/T\right)^{2q}e^{-2q}.
\end{equation}

Finally, the integral term of \eqref{s2i} can be rewritten as
\begin{equation}\label{s13}
\int_0^\infty\omega^{2q}e^{-\omega T}d\omega=\frac{1}{T}\int_0^{\infty}\left(\frac{s}{T}\right)^{2q}e^{-s}ds=\frac{1}{T^{1+2q}}\int_0^{\infty}s^{2q}e^{-s}ds=\frac{\Gamma(2q+1)}{T^{1+2q}},
\end{equation}
where by $\Gamma(\cdot)$ we indicate the Euler integral of the second kind. 

Therefore, we conclude from \eqref{s2} and \eqref{s13} that there exist two constants $C_q,C_{q,\alpha}>0$ such that 
\begin{equation}
\sum_{k=1}^\infty \omega_k^{2q}e^{-\omega_k T}\leq \frac{C_q}{T^{2q}}+\frac{C_{\alpha,q}}{T^{1+2q}}.
\end{equation}
We use this last bound to prove that there exists $\Gamma_j>0$ such that
\begin{equation}\label{ineqCM}
G_M(T)\leq\frac{M}{T^4}e^{M/T}\left[\frac{e^{M^2/(8T)}}{|\langle B\varphi_j,\varphi_j\rangle|^2}+\frac{e^{M^2/(4T)}}{b^2}\left(\frac{C_q}{T^{2q}}+\frac{C_{\alpha,q}}{T^{1+2q}}\right)\right]\leq e^{2\Gamma_j/T} \quad \forall \ T \in (0,1],
\end{equation}
as claimed.
\end{proof}
Now we proceed with the proof of Theorem \ref{Thm-suff-cond}.
\begin{proof}[Proof (of Theorem \ref{Thm-suff-cond}).]
Let $T>0$ and consider problem \eqref{newlin}. For any $y_0\in X$ and $p\in L^2(0,T)$ there exists a unique strong solution $y\in C^0([0,T],X)$ of \eqref{newlin} that can be written as
\begin{equation}
y(t)=e^{-tA}y_0-\int_0^t e^{-(t-s)A}p(s)B\varphi_jds,
\end{equation}
(see, for instance, \cite[Proposition 3.1, p. 130]{bd}).

Our aim is to find a control $p\in L^2(0,T)$ for which $y(T;y_0,0,p)=0$, that is equivalent to the following identity
\begin{equation*}
\sum_{k\in\NN^*}\langle y_0,\varphi_k\rangle e^{-\lambda_k T}\varphi_k=\int_0^T p(s)\sum_{k\in\NN^*}\langle B\varphi_j,\varphi_k\rangle e^{-\lambda_k(T-s)}\varphi_kds.
\end{equation*}
Since, by hypothesis, the eigenfunctions of $A$ form an orthonormal basis of $X$, the above formula reads as 
\begin{equation*}
\langle y_0,\varphi_k\rangle=\int_0^T e^{\lambda_ks}p(s)\langle B\varphi_j,\varphi_k\rangle ds,\quad\forall\,k\in\NN^*,
\end{equation*}
or, equivalently,
\begin{equation}\label{momp}
\int_0^T e^{\lambda_ks}p(s)ds=\frac{\langle y_0,\varphi_k\rangle}{\langle B\varphi_j,\varphi_k\rangle},\quad\forall\,k\in\NN^*.
\end{equation}
By defining $q(s):=e^{\lambda_1 s}p(s)$ and $\omega_k:=\lambda_k-\lambda_1\geq0$, the family of equations \eqref{momp} can be rewritten as
\begin{equation}\label{momp-q}
\int_0^T e^{\omega_k s}q(s)ds=\frac{\langle y_0,\varphi_k\rangle}{\langle B\varphi_j,\varphi_k\rangle},\quad\forall\,k\in\NN^*.
\end{equation}
Thanks to hypothesis \eqref{gap}, we can apply \cite[Theorem 2.4]{cmv} that ensures the existence of a biorthogonal family $\{\sigma_k\}_{k\in\NN^*}$ to the family of exponentials $\{\zeta_k\}_{k\in\NN^*}$, $\zeta_k(s)=e^{\omega_ks}$, $s\in[0,T]$.

We claim that the series
\begin{equation}\label{p}
\sum_{k\in\NN^*}\frac{\langle y_0,\varphi_k\rangle}{\langle B\varphi_j,\varphi_k\rangle}\sigma_k(s),
\end{equation}
is convergent in $L^2(0,T)$. Indeed, thanks to the following estimate, from \cite[Theorem 2.4]{cmv}, for the biorthogonal family $\{\sigma_k\}_{k\in\NN^*}$
\begin{equation*}
\norm{\sigma_k}^2_{L^2(0,T)}\leq C_\alpha(T) e^{-2\omega_kT}e^{C \sqrt{\omega_k}/\alpha},\quad\forall\,k\in\NN^*,
\end{equation*}
with $C>0$ independent of $T$ and $\alpha$, and 
\begin{equation*}
C_\alpha(T)=\left\{\begin{array}{ll}
C\left(\frac{1}{T}+\frac{1}{T^2\alpha^2}\right)e^{\frac{C}{\alpha^2 T}}&\text{if }T<\frac{1}{\alpha^2},\\\\
C^2\alpha^2&\text{if }T\geq\frac{1}{\alpha^2},
\end{array}\right.
\end{equation*}
we obtain 
\begin{equation*}
\begin{split}
\sum_{k\in\NN^*}\left|\frac{\langle y_0,\varphi_k\rangle}{\langle B\varphi_j,\varphi_k\rangle}\right|\norm{\sigma_k}_{L^2(0,T)}&\leq \norm{y_0}\left(\sum_{k\in\NN^*}\frac{\norm{\sigma_k}^2_{L^2(0,T)}}{|\langle B\varphi_j,\varphi_k\rangle|^2}\right)^{1/2}\\
&\leq\norm{y_0}\left( C^2_\alpha(T)\sum_{k\in\NN^*}\frac{e^{-2\omega_kT}e^{C\sqrt{\omega_k}/\alpha}}{|\langle B\varphi_j,\varphi_k\rangle|^2})\right)^{1/2}.
\end{split}
\end{equation*}
Observe that, by Lemma \ref{lem1}, the right-hand side of the above estimate is finite for any $T>0$.

Therefore, we define the control $q$ as
\begin{equation*}
q(s):=\sum_{k\in\NN^*}\frac{\langle y_0,\varphi_k\rangle}{\langle B\varphi_j,\varphi_k\rangle}\sigma_k(s),
\end{equation*}
and we deduce that $q\in L^2(0,T)$ satisfies \eqref{momp-q} and furthermore
\begin{equation*}
\norm{q}_{L^2(0,T)}\leq C_\alpha(T)\Lambda_T\norm{y_0},
\end{equation*}
where $\Lambda_T:=\left( \sum_{k\in\NN^*}\frac{e^{-2\omega_kT}e^{C\sqrt{\omega_k}/\alpha}}{|\langle B\varphi_j,\varphi_k\rangle|^2}\right)^{1/2}$.

Finally, returning to $p$, we obtain that
\begin{equation}
    \norm{p}^2_{L^2(0,T)}=\int_0^Te^{-2\lambda_1s}|q(s)|^2ds\leq \max\left\{1,e^{-2\lambda_1 T}\right\}\norm{q}^2_{L^2(0,T)}.
\end{equation}
By taking
\begin{equation}\label{def-control-cost-p}
    N(T):=\max\left\{1,e^{-\lambda_1 T}\right\}C_{\alpha}(T)\Lambda_T,
\end{equation}
we deduce that $\{A,B\}$ is $j$-null controllable in any time $T>0$ with associated control cost \eqref{def-control-cost-p}. 

What remains to prove is estimate \eqref{bound-control-cost} for the control cost $N(T)$ defined in \eqref{def-control-cost-p}, for $T$ small. 
Let us define $T_0$ as in \eqref{T_0}. Then for any $0<T< T_0$, it holds that
\begin{equation*}
C_\alpha(T)=C\left(\frac{1}{T}+\frac{1}{T^2\alpha^2}\right)e^{\frac{C}{\alpha^2 T}}.
\end{equation*}
We can assume without loss of generality that the constant $C \geq 1$, since we can replace it by $\max\left\{1, C\right\}$. We assume for all the sequel that $C \geq 1$.

Since $0<T< T_0 \leq 1$, we claim that there exists $\widetilde{M}>0$ such that
\begin{equation}\label{aimM}
C_\alpha^2(T) \leq \frac{\widetilde{M}}{T^4}e^{\widetilde{M}/T} \quad \forall \ T \in (0,T_0).
\end{equation}
Indeed, we have
\begin{equation*}
C_\alpha^2(T)\leq C^2\left(1+\frac{1}{\alpha^2}\right)^2\frac{1}{T^4}e^{\frac{2C}{\alpha^2T}} \quad \forall \ T \in (0,T_0).
\end{equation*}
We set
\begin{equation*}
\widetilde{M}:=C^2\left(1+\frac{1}{\alpha^2}\right)^2.
\end{equation*}
We note that since $C\geq 1$, we have
$$
\dfrac{2C}{\alpha^2} \leq \widetilde{M}.
$$
Hence from the two above estimates, we deduce \eqref{aimM}.
Moreover, we easily prove that
$$
\max\left\{1,e^{-\lambda_1 T}\right\} \leq e^{|\lambda_1|} \quad \forall \ T \in (0,T_0).
$$
Therefore, the control cost $N(T)$ given by \eqref{def-control-cost-p} can be bounded from above as follows
\begin{equation*}
N(T)\leq \sqrt{G_{M}(T)},
\end{equation*}
where $M$ is defined as in \eqref{M} and the function $G_M(\cdot)$ is defined in \eqref{G}. Finally, thanks to Lemma \ref{lem1}, we deduce that $N(T)$ fulfills property \eqref{bound-control-cost} with $\nu=\Gamma_j$.
\end{proof}

\section{Proof of Theorems \ref{teoglobal} and \ref{teoglobal0}}
Before proving Theorem \ref{teoglobal}, let us show a preliminary result that demonstrates the statement in the case of a strictly accretive operator.
\begin{lem}\label{lemglobal}
Let $A:D(A)\subset X\to X$ be a densely defined linear operator such that \eqref{ipA} holds with $\sigma=0$ and let $B:X\to X$ be a bounded linear operator. Let \{A,B\} be a 1-null controllable pair which satisfies \eqref{bound-control-cost}.  Furthermore, we assume $\lambda_1=0$. Then, there exists a constant $r_1>0$ such that for any $R>0$ there exists $T_{R}>0$ such that for all $v_0\in X$ that satisfy
\begin{equation}\label{ipv0}
\left|\langle v_0,\varphi_1\rangle\right| <  r_1,\qquad\norm{v_0-\langle v_0,\varphi_1\rangle\varphi_1}\leq R,
\end{equation}
problem \eqref{v} is null controllable in time $T_{R}$.
\end{lem}
\begin{proof}
{\bf First step.} We fix $T=1$. Thanks to Theorem \ref{teo1}, there exists a constant $r_1>0$ such that, denoting by $u_1$ the solution of  \eqref{u} on $[0,1]$, if $\norm{u_1(0)- \varphi_1}< \sqrt{2}r_1$ then there exists a control $p_1\in L^2(0,1)$ for which the solution of \eqref{u} of with $p$ replaced by $p_1$, satisfies $u_1(1)=\varphi_1$. We set $v_1=u_1- \varphi_1$ on $[0,1]$. We deduce that
if $\norm{v_1(0)} < \sqrt{2}r_1$ then there exists a control $p_1\in L^2(0,1)$ for which the solution $v_1$ of \eqref{v} on $[0,1]$ with $p$ replaced by $p_1$, satisfies $v_1(1)=0$.

\smallskip\noindent
{\bf Second step.} Let $v_0\in X$ be the initial condition of \eqref{v}. We decompose $v_0$ as follows
\begin{equation*}
v_0=\langle v_0,\varphi_1\rangle\varphi_1+v_{0,1},
\end{equation*}
where $v_{0,1}\in \varphi_1^\perp$ and we suppose that $\left|\langle v_0,\varphi_1\rangle\right| < r_1$. If $R\leq r_1$, then $ \norm{v_0}^2 \leq r^2_1+R^2\leq 2r^2_1$ and we can directly apply the first step of the proof
with $T_R=1$. Otherwise, we define $t_{R}$ as
\begin{equation}\label{trR}
t_{R}:=\frac{1}{2\lambda_2}\log{\left(\frac{R^2}{r_1^2}\right)},
\end{equation} 
and in the time interval $[0,t_{R}]$ we take the control $p\equiv0$. Then, for all $t\in [0,t_{R}]$, we have that
\begin{equation*}
\norm{v(t)}^2\leq \norm{e^{-tA}\left(\langle v_0,\varphi_1\rangle\varphi_1+v_{0,1}\right)}^2\leq \left|\langle v_0,\varphi_1\rangle\right|^2+e^{-2\lambda_2 t}\norm{v_{0,1}}^2 < r_1^2+e^{-2\lambda_2 t}R^2.
\end{equation*}
In particular, for $t=t_{R}$, it holds that $\norm{v(t_{R})}^2 < 2 r^2_1$.

Now, we define $T_{R}:=t_{R}+1$ and set $v_1(0)=v(t_R)$. Thanks to the first step of the proof, there exists a control $p_1\in L^2(0,1)$, such that $v_1(1)=0$, where $v_1$ is the solution of \eqref{v} on $[0,1]$ with $p$ replaced by $p_1$.

Then $v(t)=v_1(t-t_R)$ solves \eqref{v} in the time interval $(t_{R},T_{R}]$ with the control $p_1(t-t_{R})$ that steers the solution $v$ to $0$ at $T_{R}$.
\end{proof}
\begin{proof}[Proof (of Theorem \ref{teoglobal})]
We start with the case $\lambda_1=0$. Let $u_0\in X$ satisfy \eqref{ipu0}. Set $v(t):=u(t)-\varphi_1$, then $v$ satisfies \eqref{v} and moreover $v_0:=v(0)=u_0-\varphi_1$ fulfills \eqref{ipv0}. Thus, by Lemma \ref{lemglobal}, problem \eqref{u} is exactly controllable to the first eigensolution $\psi_1
\equiv\varphi_1$ in time $T_{R}$.

Now, we consider the case $\lambda_1>0$. As in the proof of Theorem \ref{teo1}, we introduce the variable $z(t)=e^{\lambda_1t}u(t)$ that solves problem \eqref{z}. For such a system, since the first eigenvalue of $A_1$ is equal $0$, we have the exact controllability to $\varphi_1$ in time $T_{R}$. Namely $z(T_{R})=\varphi_1$, that is equivalent to the exact controllability of $u$ to $\psi_1$:

\begin{equation}
z(T_{R})=\varphi_1\
\quad\Longleftrightarrow\quad
e^{\lambda_1T_{R}}u(T_{R})=\varphi_1
\quad\Longleftrightarrow\quad
u(T_{R})=\psi_1(T_{R}).
\end{equation}
The proof is thus complete.
\end{proof}

The proof of Theorem \ref{teoglobal0} easily follows from Theorem \eqref{teoglobal}.
\begin{proof}[Proof (of Theorem \ref{teoglobal0})]
We assume \eqref{cone}. Suppose that $\gamma:=\langle u_0,\varphi_1\rangle\neq0$. We decompose $u_0$ as $u_0=\gamma\varphi_1+\zeta_1$, with $\zeta_1:=u_0-\langle u_0,\varphi_1\rangle\varphi_1\in\varphi_1^\perp$ and define $\tilde{u}(t):=u(t)/\gamma$. Hence, $\tilde{u}$ solves
\begin{equation}\label{utildeglobal}
\left\{
\begin{array}{ll}
\tilde{u}'(t)+A\tilde{u}(t)+p(t)B\tilde{u}(t)=0,& t>0\\
\tilde{u}(0)=\varphi_1+\tilde{\zeta_1},
\end{array}\right.
\end{equation}
where $\tilde{\zeta_1}:=\zeta_1/\gamma$.

We apply Theorem \ref{teoglobal} to \eqref{utildeglobal} to deduce the existence of $T_R>0$ such that $\tilde{u}(T_R)=\psi_1(T_R)$. Therefore, the solution of \eqref{u} with initial condition $u_0\in X$ that do not vanish along the direction $\varphi_1$ can be exactly controlled in time $T_R$ to the trajectory $\phi_1(\cdot)=\langle u_0,\varphi_1\rangle\psi_1(\cdot)$, where
$\phi_1$ is defined in \eqref{exactphi1}.

Note that if $u_0\in X$ satisfies both $u_0\in\varphi_1^\perp$ and \eqref{cone}, then we have trivially that $u_0\equiv 0$. We then choose $p\equiv 0$, so that the solution of \eqref{u} remains constantly equal to $\phi_1\equiv 0$.
\end{proof}

\section{Applications}
In this section we present some examples of parabolic equations for which Theorem \ref{teo1} can be applied. The hypotheses \eqref{ipA},\eqref{gap} and \eqref{ipB} have been verified in \cite{acu} and \cite{cu}, to which we refer for more details. We observe that, thanks to \cite[Remark 6.1]{acu}, since the second order operators considered in the examples are accretive ($\langle Ax,x\rangle\geq0$, for all $x\in D(A)$), it suffices to prove the following gap condition
\begin{equation}\label{gap-no-lambda1}
    \exists\,\alpha>0\,:\,\sqrt{\lambda_{k+1}}-\sqrt{\lambda_k}\geq\alpha,\quad\forall\,k\geq1,
\end{equation}
which implies \eqref{gap}.

Moreover, in the case of an accretive operator it suffices to show that there exist $b,q>0$ such that
\begin{equation}\label{ipB-accr}
\begin{array}{l}
\langle B\varphi_j,\varphi_j\rangle\neq0\quad\mbox{and}\quad\left|\lambda_k-\lambda_j\right|^q|\langle B\varphi_j,\varphi_k\rangle|\geq b,\quad\forall\,k\neq j.
\end{array}
\end{equation}
to have \eqref{ipB}.

Furthermore, we note that the global results Theorem \ref{teoglobal} and Theorem \ref{teoglobal0} can be applied to any example below. Note also, that the given examples below, are non-exhaustive.
\subsection{Diffusion equation with Dirichlet boundary conditions}\label{ex1}
Let $I=(0,1)$ and $X=L^2(0,1)$. Consider the following problem
\begin{equation}\label{eqex1}\left\{\begin{array}{ll}
u_t(t,x)-u_{xx}(t,x)+p(t)\mu(x)u(t,x)=0 & x\in I,t>0 \\\\
u(t,0)=0,\,\,u(t,1)=0, & t>0\\\\
u(0,x)=u_0(x) & x\in I.
\end{array}\right.
\end{equation}
We denote by $A$ the operator defined by
\begin{equation*}
D(A)=H^2\cap H^1_0(I),\quad A\varphi=-\frac{d^2\varphi}{dx^2}.
\end{equation*}
and it can be checked that $A$ satisfies \eqref{ipA}. We indicate by $\{\lambda_k\}_{k\in\NN^*}$ and $\{\varphi_k\}_{k\in\NN^*}$ the families of eigenvalues and eigenfunctions of $A$, respectively:
\begin{equation*}
\lambda_k=(k\pi)^2,\quad \varphi_k(x)=\sqrt{2}\sin(k\pi x),\quad \forall\, k\in\NN^*.
\end{equation*}

It is easy to see that \eqref{gap-no-lambda1} holds true (and so \eqref{gap}):
\begin{equation*}
\sqrt{\lambda_{k+1}}-\sqrt{\lambda_k}=\pi,\qquad \forall\, k\in \NN^*.
\end{equation*}

Let $B:X\to X$ be the operator
\begin{equation*}
B\varphi=\mu\varphi
\end{equation*}
with $\mu\in H^3(I)$ such that
\begin{equation}\label{mu}
\mu'(1)\pm\mu'(0)\neq 0\quad\mbox{ and }\quad\langle\mu\varphi_j,\varphi_k\rangle\neq0\quad\forall\, k \in \NN^*.
\end{equation}
Then, there exists $b>0$ such that
\begin{equation*}
\left|\lambda_k-\lambda_j\right|^{3/2}|\langle \mu\varphi_j,\varphi_k\rangle|\geq b,\qquad\forall\, k\in\NN^*.
\end{equation*}
For instance, a suitable function that satisfies \eqref{mu} is $\mu(x)=x^2$: indeed, in this case 
\begin{equation*}
\langle \mu\varphi_j,\varphi_k\rangle=\left\{\begin{array}{ll}
\frac{4kj(-1)^{k+j}}{(k^2-j^2)^2},&  k\neq j,\\\\
\frac{2j^2\pi^2-3}{6j^2\pi^2},&k=j.
\end{array}\right.
\end{equation*}
Therefore, problem \eqref{eqex1} is controllable to the $j$-th eigensolution $\psi_j$ in any time $T>0$ as long as $u_0\in B_{R_T}(\varphi_j)$, with $R_T>0$ a suitable constant, where $\psi_j(t,x)=\sqrt{2}\sin(j\pi x)e^{-j^2\pi^2t}$.
\subsection{Diffusion equation with Neumann boundary conditions}\label{ex2}
Let $I=(0,1)$, $X=L^2(I)$ and consider the Cauchy problem
\begin{equation}\label{eqex2}
\left\{\begin{array}{ll}
u_t(t,x)-u_{xx}(t,x)+p(t)\mu(x)u(t,x)=0 & x\in I,t>0 \\\\
u_x(t,0)=0,\,\,u_x(t,1)=0, &t>0\\\\
u(0,x)=u_0(x). & x\in I.
\end{array}\right.
\end{equation}
The operator $A$, defined by
\begin{equation*}
D(A)=\{ \varphi\in H^2(0,1): \varphi'(0)=0,\,\,\varphi'(1)=0\},\quad A\varphi=-\frac{d^2\varphi}{dx^2}
\end{equation*}
satisfies \eqref{ipA} and has the following eigenvalues and eigenfunctions 
\begin{equation*}
\begin{array}{lll}
\lambda_0=0,&\varphi_0=1\\
\lambda_k=(k\pi)^2,& \varphi_k(x)=\sqrt{2}\cos(k\pi x),& \forall\, k\geq1.
\end{array}
\end{equation*}
Thus, the gap condition \eqref{gap-no-lambda1} is fulfilled with $\alpha=\pi$. Fixed $j\in\NN$, the $j$-th eigensolution is the function $\psi_j(x)=e^{-\lambda_j t}\varphi_j(x)$.

We define $B:X\to X$ as the multiplication operator by a function $\mu\in H^2(I)$, $B\varphi=\mu\varphi$, such that
\begin{equation}
\mu'(1)\pm\mu'(0)\neq 0\quad\mbox{ and }\quad\langle\mu\varphi_j,\varphi_k\rangle\neq0\quad\forall\, k \in \NN.
\end{equation}
It can be proved that, there exists $b>0$ such that
\begin{equation}\label{mu2}
\left|\lambda_k-\lambda_j\right||\langle \mu\varphi_j,\varphi_k\rangle|\geq b,\qquad\forall\, k\in\NN^*.
\end{equation}
For example, $\mu(x)=x^2$ satisfies \eqref{mu2}. Indeed, it can be shown that
\begin{equation*}
\langle \mu\varphi_0,\varphi_k\rangle=\left\{\begin{array}{ll}
\frac{2\sqrt{2}(-1)^{k}}{(k\pi)^2},&k\geq1,\\\\
\frac{1}{3},&k=0,
\end{array}\right.
\end{equation*}
and for $j\neq0$
\begin{equation*}
\langle \mu\varphi_j,\varphi_k\rangle=\left\{\begin{array}{ll}
\frac{4(-1)^{k+j}(k^2+j^2)}{(k^2-j^2)^2\pi^2},&k\neq j,\\\\
\frac{1}{3}+\frac{1}{2j^2\pi^2},&k=j.
\end{array}\right.
\end{equation*}

Therefore, problem \eqref{eqex2} is controllable to the  $j$-th eigensolution $\psi_j$ in any time $T>0$ as long as $u_0\in B_{R_T}(\varphi_j)$, with $R_T>0$ a suitable constant.
\subsection{Variable coefficient parabolic equation}\label{ex3}
Let $I=(0,1)$, $X=L^2(I)$ and consider the problem
\begin{equation}\label{eqex3}
\left\{
\begin{array}{ll}
u_t(t,x)-((1+x)^2u_x(t,x))_x+p(t)\mu(x)u(t,x)=0&x\in I,t>0\\\\
u(t,0)=0,\quad u(t,1)=0,&t>0\\\\
u(0,x)=u_0(x)&x\in I.
\end{array}
\right.
\end{equation}
We denote by $A:D(A)\subset X\to X$ the following operator
\begin{equation*}
D(A)=H^2\cap H^1_0(I),\qquad A\varphi=-((1+x)^2\varphi_x)_x.
\end{equation*}
It can be checked that $A$ satisfies \eqref{ipA} and that the eigenvalues and eigenfunctions have the following expression
\begin{equation*}
\lambda_k=\frac{1}{4}+\left(\frac{k\pi}{\ln2}\right)^2,\qquad\varphi_k=\sqrt{\frac{2}{\ln 2}}(1+x)^{-1/2}\sin\left(\frac{k\pi}{\ln2 }\ln(1+x)\right).
\end{equation*}
Furthermore, $\{\lambda_k\}_{k\in\NN^*}$ verifies the gap condition \eqref{gap-no-lambda1} with $\alpha=\pi/\ln{2}$.

We fix $j\in\NN^*$ and define the operator $B:X\to X$ by $B\varphi=\mu\varphi$, where $\mu\in H^2(I)$ is such that
\begin{equation}\label{mu3}
2\mu'(1)\pm\mu'(0)\neq0,\quad\mbox{and}\quad \langle \mu\varphi_j,\varphi_k\rangle\neq0\quad\forall k \in\NN^*.
\end{equation}
Hence, thanks to \eqref{mu3}, it is possible to show that \eqref{ipB-accr} is fulfilled with $q=3/2$ (see \cite[Section 6.3]{acu}). For instance, when $j=1$, an example of a suitable function $\mu$ that satisfies \eqref{mu3} is $\mu(x)=x$, see \cite{acu} for the verification.

Thus, from Theorem \ref{teo1}, we deduce that, for any $T>0$, system \eqref{eqex3} is controllable to the $j$-th eigensolution if the initial condition $u_0$ is close enough to $\varphi_j$. 
\subsection{Diffusion equation in a $3D$ ball with radial data}\label{ex4}
In this example, we study the controllability of an evolution equation in the three dimensional unit ball $B^3$ for radial data. The bilinear control problem is the following
\begin{equation}\label{eqex4}
\left\{\begin{array}{ll}
u_t(t,r)-\Delta u(t,r)+p(t)\mu(r)u(t,r)=0 & r\in[0,1], t>0 \\\\
u(t,1)=0,&t>0\\\\
u(0,r)=u_0(r) & r\in[0,1]
\end{array}\right.
\end{equation}
where the Laplacian in polar coordinates for radial data is given by the following expression
$$\Delta\varphi(r)=\partial^2_r \varphi(r)+\frac{2}{r}\partial_r\varphi(r).$$
The function $\mu$ is a radial function as well in the space $H^3_r(B^3)$, where the spaces $H^k_r(B^3)$ are defined as follows
$$X:=L^2_{r}(B^3)=\left\{\varphi\in L^2(B^3)\,|\, \exists \psi:\RR\to\RR, \varphi(x)=\psi(|x|)\right\}$$
$$H^k_r(B^3):=H^k(B^3)\cap L^2_{r}(B^3) .$$

The domain of the Dirichlet Laplacian $A:=-\Delta$ in $X$ is $D(A)=H^2_{r}\cap H^1_0(B^3)$. We observe that $A$ satisfies hypothesis \eqref{ipA}. We denote by $\{\lambda_k\}_{k\in\NN^*}$ and $\{\varphi_k\}_{k\in\NN^*}$ the families of eigenvalues and eigenvectors of $A$, $A\varphi_k=\lambda_k\varphi_k$, namely
\begin{equation}\label{ee}\varphi_k=\frac{\sin(k\pi r)}{\sqrt{2\pi}r},\qquad\lambda_k=(k\pi)^2
\end{equation}
$\forall\, k\in\NN^*$, see \cite[Section 8.14]{leb}. Since the eigenvalues of $A$ are actually the same of the Dirichlet $1D$ Laplacian, \eqref{gap-no-lambda1} is satisfied, as we have seen in Example \ref{ex1}.

Fixed $j\in\NN^*$, let $B:X\to X$ be the multiplication operator $Bu(t,r)=\mu(r)u(t,r)$, with $\mu$ be such that
\begin{equation}\label{mu4}
\mu'(1)\pm\mu'(0)\neq 0,\quad\mbox{and}\quad \langle \mu\varphi_j,\varphi_k\rangle\neq0\quad\forall\, k\in\NN^*.
\end{equation}

Then, it can be proved that 
\begin{equation}\label{mu4p}
\left|\lambda_k-\lambda_j\right|^{3/2}|\langle \mu\varphi_j,\varphi_k\rangle|\geq b,\qquad \forall\, k\in\NN^*,
\end{equation}
with $b$ a positive constant. For instance, $\mu(x)=x^2$ verifies \eqref{mu4} and \eqref{mu4p}:
\begin{equation*}
\langle B\varphi_j,\varphi_k\rangle=\left\{\begin{array}{ll}
\frac{4(-1)^{k+j}kj}{(k^2-j^2)^2\pi^2},&k\neq j,\\\\
\frac{2j^2\pi^2-3}{6j^2\pi^2},&k=j.
\end{array}\right.
\end{equation*}

Therefore, by applying Theorem \ref{teo1}, we conclude that for any $T>0$, the exists a suitable constant $R_T>0$ such that, if $u_0\in B_{R_T}(\varphi_j)$, problem \eqref{eqex4} is exactly controllable to the $j$-th eigensolution $\psi_j$ in time $T$.

\subsection{Degenerate parabolic equation}\label{ex5}
In this last section we want to address an example of a control problem for a degenerate evolution equation of the form
\begin{equation}\label{eqex5}
\left\{
\begin{array}{ll}
u_t-\left(x^{\gamma} u_x\right)_x+p(t)x^{2-\gamma}u=0,& (t,x)\in (0,+\infty)\times(0,1)\\\\
u(t,1)=0,\quad\left\{\begin{array}{ll} u(t,0)=0,& \mbox{ if }\gamma\in[0,1),\\\\ \left(x^{\gamma}u_x\right)(t,0)=0,& \mbox{ if }\gamma\in[1,3/2),\end{array}\right.\\\\
u(0,x)=u_0(x).
\end{array}
\right.
\end{equation}
where $\gamma\in[0,3/2)$ describes the degeneracy magnitude, for which Theorem \ref{teo1} applies. 

If $\gamma\in[0,1)$ problem \eqref{eqex5} is called weakly degenerate and the natural spaces for the well-posedness are the following weighted Sobolev spaces. Let $I=(0,1)$ and $X=L^2(I)$, we define
\begin{equation*}
\begin{array}{l}
H^1_{\gamma}(I)=\left\{u\in X: u \mbox{ is absolutely continuous on } [0,1], x^{\gamma/2}u_x\in X\right\}\\\\
H^1_{\gamma,0}(I)=\left\{u\in H^1_\gamma(I):\,\, u(0)=0,\,\,u(1)=0\right\}\\\\
H^2_\gamma(I)=\left\{u\in H^1_\gamma(I): x^{\gamma}u_x\in H^1(I)\right\}.
\end{array}
\end{equation*}
We denote by $A:D(A)\subset X\to X$ the linear degenerate second order operator
\begin{equation}
\left\{\begin{array}{l}
\forall u\in D(A),\quad Au:=-(x^{\gamma}u_x)_x,\\\\
D(A):=\{u\in H^1_{\gamma,0}(I),\,\, x^{\gamma}u_x\in H^1(I)\}.
\end{array}\right.
\end{equation}
It is possible to prove that $A$ satisfies \eqref{ipA} (see, for instance \cite{cmp}) and furthermore, if we denote by $\{\lambda_k\}_{k\in\NN^*}$  the eigenvalues and by $\{\varphi_k\}_{k\in\NN^*}$ the corresponding eigenfunctions, it turns out that the gap condition \eqref{gap-no-lambda1} is fulfilled with $\alpha=\frac{7}{16}\pi$ (see \cite{kl}, page 135). 

If $\gamma\in[1,2)$, problem \eqref{eqex5} is called strong degenerate and the corresponding weighted Sobolev space are described as follows: given $I=(0,1)$ and $X=L^2(I)$, we define
\begin{equation*}
\begin{array}{l}
H^1_{\gamma}(I)=\left\{u\in X: u \mbox{ is absolutely continuous on } (0,1],\,\, x^{\gamma/2}u_x\in X\right\}\vspace{.1cm}\\\\
H^1_{\gamma,0}(I):=\left\{u\in H^1_{\gamma}(I):\,\,u(1)=0\right\},\vspace{.1cm}\\\\
H^2_\gamma(I)=\left\{u\in H^1_\gamma(I):\,\, x^{\gamma}u_x\in H^1(I)\right\}.
\end{array}
\end{equation*}
In this case the operator $A:D(A)\subset X\to X$ is defined by
\begin{equation*}
\left\{\begin{array}{l}
\forall u\in D(A),\quad Au:=-(x^{\gamma}u_x)_x,\vspace{.1cm}\\\\
D(A):=\left\{u\in H^1_{\gamma,0}(I):\,\, x^{\gamma}u_x\in H^1(I)\right\}\vspace{.1cm}\\
\qquad\,\,\,\,\,=\left\{u\in X:\,\,u \mbox{ is absolutely continuous in (0,1] },\,\, x^{\gamma}u\in H^1_0(I),\right.\vspace{.1cm}\\
\qquad\qquad\,\,\,\left.x^{\gamma}u_x\in H^1(I)\mbox{ and } (x^{\gamma}u_x)(0)=0\right\}
\end{array}\right.
\end{equation*}
and it has been proved that \eqref{ipA} holds true (see, for instance \cite{cmvn}) and that \eqref{gap-no-lambda1} is satisfied for $\alpha=\frac{\pi}{2}$ (see \cite{kl}).

We fix $j=1$ and for all $\gamma\in[0,3/2)$, we define the linear operator $B:X\to X$ by $Bu(t,x)=x^{2-\gamma}u(t,x)$ and in \cite{cu} we have proved that there exists a constant $b>0$ such that
\begin{equation*}
\left|\lambda_k-\lambda_1\right|^{3/2}|\langle B\varphi_1,\varphi_k\rangle|\geq b\quad\forall k\in\NN^*.
\end{equation*}

Finally, by applying Theorem \ref{teo1}, we ensure the exact controllability of problem \eqref{eqex5} to the first eigensolution, for both weakly and strongly degenerate problems.

\section*{Acknowledgments}
We are grateful to J. M. Coron and P. Martinez for their precious comments and suggestions.

\appendix
\section{Proofs of Propositions \ref{prop38} and \ref{prop39}}\label{appendix}
In this Appendix we present the proofs of Propositions \ref{prop38} and \ref{prop39} which follow those of \cite[Proposition 4.3]{acu} and \cite[Proposition 4.4]{acu}, respectively. However, since the hypotheses of Propositions \ref{prop38} and \ref{prop39} are slightly different with respect to \cite[Proposition 4.3]{acu} and \cite[Proposition 4.4]{acu}, we preferred to show their proofs in the current settings.
\begin{proof}[Proof of Proposition \ref{prop38}.]
Thanks to Remark \ref{rmk-regularity}, taking the scalar product of \eqref{v} with $v$, we obtain
\begin{equation}
\langle v'(t),v(t)\rangle+\langle A v(t),v(t)\rangle+p(t)\langle Bv(t)+B\varphi_j,v(t)\rangle=0,\quad\text{for a.e. }t\in[\eps,T].
\end{equation}
Thus, we get that for a.e. $t\in[\eps,T]$
\begin{equation}\label{energy}
\begin{split}
\frac{1}{2}\frac{d}{dt}\norm{v(t)}^2+\langle A v(t),v(t)\rangle&\leq  \norm{B}\left(|p(t)|\norm{v(t)}^2+|p(t)|\norm{\varphi_j}\norm{v(t)}\right)\\
&\leq \norm{B}\left(|p(t)|\norm{v(t)}^2+\frac{1}{2}|p(t)|^2+\frac{1}{2}\norm{v(t)}^2\right).\\
\end{split}
\end{equation}
Therefore, since $A$ satisfies \eqref{ipAA}, we have that
\begin{equation*}
\frac{1}{2}\frac{d}{dt}\norm{v(t)}^2\leq \left(\sigma+\norm{B}|p(t)|+\frac{\norm{B}}{2}\right)\norm{v(t)}^2+\frac{1}{2}\norm{B}|p(t)|^2,\quad \text{for a.e. }t\in[\eps,T].
\end{equation*}
We now integrate the last inequality from $\eps$ to $t$ to obtain
\begin{multline*}
\int_\eps^t \frac{d}{ds}\norm{v(s)}^2ds\leq\\ \int_\eps^t\left(2\sigma+\norm{B}(2|p(s)|+1)\right)\norm{v(s)}^2ds+\norm{B}\int_0^{T}|p(s)|^2ds,\quad \text{for a.e. }t\in[\eps,T],
\end{multline*}
and, by Gronwall's inequality, we conclude that
\begin{equation*}
\norm{v(t)}^2\leq\left(\norm{v(\eps)}^2+\norm{B}\int_0^{T}|p(s)|^2ds\right)e^{\int_\eps^t(2\sigma+\norm{B}(2|p(s)|+1))}ds,\quad \text{for a.e. }t\in[\eps,T].
\end{equation*}
Taking the limit as $\eps\to0$, we find
\begin{equation*}
\norm{v(t)}^2\leq\left(\norm{v_0}^2+\norm{B}\int_0^{T}|p(s)|^2ds\right)e^{\int_0^t(2\sigma+\norm{B}(2|p(s)|+1))}ds,\quad \text{for a.e. }t\in[0,T].
\end{equation*}
Thus, taking the supremum over the interval $[0,T]$, the last inequality becomes
\begin{equation*}
\begin{split}
\sup_{t\in[0,T]}\norm{v(t)}^2\leq e^{2\norm{B}\sqrt{T}\norm{p}_{L^2(0,T)}+(2\sigma+\norm{B})T}\left(\norm{v_0}^2+\norm{B}\norm{p}^2_{L^2(0,T)}\right).
\end{split}
\end{equation*}
Finally, recalling estimate \eqref{prelim-p-bound} for $p$, we get \eqref{unifv}.
\end{proof}

\begin{proof}[Proof of Proposition \ref{prop39}.]
Since $w\in H^1(0,T;X)\cap L^2(0,T;D(A))$, taking the scalar product of both members of the equation in \eqref{w} with $w(t)$, we obtain
\begin{equation}\begin{split}
\frac{1}{2}\frac{d}{dt}\norm{w(t)}^2&\leq \sigma\norm{w(t)}^2+|p(t)|\norm{Bv(t)}\norm{w(t)}\\
&\leq \left(\frac{1}{2}+\sigma\right)\norm{w(t)}^2+\norm{B}^2\frac{1}{2}|p(t)|^2\norm{v(t)}^2,\quad \text{for a.e }t\in[0,T].
\end{split}
\end{equation}

Then, by Gronwall's inequality, and appealing to \eqref{unifv} and \eqref{prelim-p-bound}, we get
\begin{equation}\begin{split}
\sup_{t\in[0,T]}\norm{w(t)}^2&\leq \norm{B}^2e^{(2\sigma+1)T}\norm{p}^2_{L^2(0,T)}\sup_{t\in[0,T]}\norm{v(t)}^2\\
&\leq \norm{B}^2e^{(4\sigma+\norm{B}+1)T+2\norm{B}N(T)\sqrt{T}\norm{v_0}}(1+\norm{B}N(T)^2)\norm{v_0}^2\norm{p}^2_{L^2(0,T)}\\
&\leq \norm{B}^2N(T)^2e^{(4\sigma+\norm{B}+1)T+2\norm{B}N(T)\sqrt{T}\norm{v_0}}(1+\norm{B}N(T)^2)\norm{v_0}^4.
\end{split}
\end{equation}
Using \eqref{v0}, we obtain that
\begin{equation*}
\sup_{t\in[0,T]}\norm{w(t)}^2\leq K(T)^2\norm{v_0}^4,
\end{equation*}
which yields to \eqref{wT} and \eqref{KT}.
\end{proof}

\section*{References}
\bibliographystyle{plain}
\bibliography{biblio}
\end{document}